\documentclass[11pt,leqno]{amsart}

\usepackage{amssymb}
\usepackage{amsmath}
\usepackage{amsthm}
\usepackage{latexsym}
\usepackage{amscd, xypic}
\usepackage[dvipsnames]{xcolor}
\usepackage{tikz-cd} 
\tikzcdset{arrow style=tikz}

\begin{document}

\newtheorem{thm}{Theorem}
\newtheorem{prop}[thm]{Proposition}
\newtheorem{lem}[thm]{Lemma}
\newtheorem{cor}[thm]{Corollary}
\newtheorem*{que}{Question}
\newtheorem*{ass}{Assumption}
\theoremstyle{definition}
\newtheorem{defi}[thm]{Definition}
\newtheorem{ex}[thm]{Example}
\newtheorem{rem}[thm]{Remark}

\newcommand{\iso}{\mbox{Iso}}
\newcommand{\m}{^{-1}}
\newcommand{\G}{\mathcal {G}}
\newcommand{\U}{\mathcal {U}}
\newcommand{\N}{\mathbb {N}}
\newcommand{\si}{\sigma}
\newcommand{\rh}{\rho}
\newcommand{\ta}{\tau}
\newcommand{\rmod}{R\text{-\textup{\textbf{mod}}}}
\newcommand{\rumod}{R\text{-\textup{\textbf{umod}}}}
\newcommand{\grmod}{\G\textup{-}\rmod}
\newcommand{\grumod}{\G\textup{-}\rumod}
\newcommand{\umodgr}{\G\text{-\textup{\textbf{umod}}}\textup{-}R}
\newcommand\id[1]{\mathbf{1}_{#1}} 
\newtheorem{assertion}[thm]{{\sc Assertion}}
\newcommand\restr[2]{{
  \left.\kern-\nulldelimiterspace 
  #1 
  \vphantom{\big|} 
  \right|_{#2} 
  }}

\title[Object-unital groupoid graded rings]{Object-unital groupoid graded rings,
crossed products and separability}

\author[J. Cala]{Juan  Cala }
\address{Escuela de Matem\'{a}ticas,
Universidad Industrial de Santander,
Carrera 27 Calle 9,
Edificio Camilo Torres
Apartado de correos 678,
Bucaramanga, Colombia}
\email{{\scriptsize jcalab@gmail.com }}

\author[P. Nystedt]{Patrik Nystedt}
\address{University West,
Department of Engineering Science, 
SE-46186 Trollh\"{a}ttan, Sweden}
\email{{\scriptsize patrik.nystedt@hv.se}}

\author[H. Pinedo]{Hector Pinedo}
\address{Escuela de Matem\'{a}ticas,
Universidad Industrial de Santander,
Carrera 27 Calle 9,
Edificio Camilo Torres
Apartado de correos 678,
Bucaramanga, Colombia}
\email{{\scriptsize hpinedot@uis.edu.co}}



\begin{abstract}
We extend the classical construction by Noether of crossed product algebras, defined by finite
Galois field extensions, to cover the case of separable (but not necessarily finite or normal) field extensions.
This leads us naturally to consider non-unital groupoid graded rings of a particular type that we call object unital.
We determine when such rings are strongly graded, crossed products, skew groupoid rings and twisted groupoid rings. 
We also obtain necessary and sufficient criteria for when object unital groupoid graded rings are 
separable over their principal component, thereby generalizing 
previous results from the unital case to a non-unital situation.
\end{abstract}

\maketitle

\section{Introduction}\label{section:introduction}

Recall that if $L/K$ is a finite Galois field extension with Galois group $G$,
$$G \ni \si \mapsto \alpha_\si \in {\rm Aut}_K(L)$$
denotes the corresponding group homomorphism and 
$$G \times G \ni (\si,\ta) \mapsto \beta_{\si,\ta} \in L \setminus \{ 0 \}$$
is a {\it cocycle}, that is for all $\si,\ta,\rho \in G$ the relations
\begin{equation}\label{identity}
\beta_{\si , e} = \beta_{ e , \si} = 1,
\end{equation}
where $e$ denotes the multiplicative identity of $G$, and
\begin{equation}\label{cocycle} 
\beta_{\si,\ta} \beta_{\si\ta,\rho} = \alpha_\si( \beta_{\ta,\rho} ) \beta_{\si,\ta\rho}
\end{equation}
hold, then we can define the so called {\it crossed product algebra} $(L/K,\beta)$
in the following way.
As a left $L$-vector space this structure is defined to be $\oplus_{\si \in G} L u_\si$
where $\{ u_\si \}_{\si \in G}$ is a set of symbols.
The multiplication in $(L/K,\beta)$
is defined by the $K$-linear extension of the relations
\begin{equation}\label{multiplication}
( x u_\si ) ( y u_\tau ) = x \alpha_\si(y) \beta_{\si,\ta} u_{\si\ta},
\end{equation}
for $x,y \in L$ and $\si , \ta \in G$. 
It is well known that $(L/K,\beta)$
is a central simple $K$-algebra (see e.g. \cite[Theorem (29.6)]{reiner1975}).

The crossed product algebras can be put in the more general context of graded rings in the following sense.
If we put $R = (L/K,\beta)$ 
and $R_\si = L u_\si$, for $\si \in G$, then
$R = \oplus_{\si \in G} R_\si$ and $R_\si R_\tau \subseteq R_{\si\ta}$, for $\sigma,\tau \in G$, 
which means that $R$ is {\it graded} by $G$ (or $G$-graded).
In fact, since the $\beta_{\si,\ta}$ are non-zero, we always have equality 
$R_\si R_\ta = R_{\si\tau}$, for $\si,\tau \in G$, so that $R$ is a so called {\it strongly graded} ring.

The impetus for the theoretical framework in this article is the following

\begin{que}\label{mainquestion}
Is it possible to define a crossed product-like structure, starting from a separable,
but not necessarily finite or Galois, field extension $L/K$?
\end{que}

In \cite{lundstrom2005} the second author of the present article suggested an answer to this question 
in the case when $L/K$ is a {\it finite} separable field extension.
In doing so he was naturally lead to structures which no longer, in any natural sense,
fitted into the framework of group graded rings, bur instead suited well in the context of rings graded by {\it groupoids}.
Let us briefly describe these structures.

A groupoid $\G$ is a small category with the property that all morphisms are isomorphisms. 
Equivalently, this can be defined by saying that $\G$ is a set equipped with
a unary operation $\G \ni \sigma \mapsto \sigma^{-1} \in \G$ (inversion) and a partially defined multiplication
$\G \times \G \ni (\sigma,\tau) \mapsto \sigma \tau \in \G$ (composition) such that
for all $\si,\ta,\rh \in \G$ the following four axioms hold:
(i) $(\si^{-1})^{-1} = \si$;
(ii) if $\si \tau$ and $\ta \rh$ are defined, then $(\si \ta) \rh$ and $\si (\ta \rh)$ are defined and $(\si \ta) \rh=\si (\ta \rh)$;
(iii) the {\it domain} $d(\si) := \si^{-1} \si$ is always defined and if $\si \tau$ is defined, then $d(\si) \tau = \tau$;
(iv) the {\it range} $r(\ta) := \ta \ta^{-1}$ is always defined and if $\si \tau$ is defined, then $\si r(\ta) = \si$.
The maps $d$ and $r$ have a common image denoted by $\G_0$, which is called the \textit{unit space} of $\G$.
The set $\G_2 = \{ (\si,\ta) \in \G \times \G \mid \mbox{$\sigma \tau$ is defined} \}$ 
is called the set of {\it composable pairs} of $\G$ and, analogously,
$\G_3$ denotes the set of {\it composable triples} of $\G$, 
that is $\{ (\si,\ta,\rho) \in \G \times \G \times \G \mid (\si,\ta) \in \G_2 \ {\rm and} \ (\ta,\rho) \in \G_2 \}$. 
For more details about groupoids, the interested reader may consult for example \cite{L} or \cite{ren}.  

Suppose that $R$ is an associative but not necessarily unital ring.
If $R$ is unital, then we let $1_R$ denote the multiplicative identity of $R$.
Let $\G$ be a groupoid. Recall from \cite{lundstrom2004} (see also \cite{liu2005}) that
the ring $R$ is said to be \emph{graded by $\G$}, or \emph{$\G$-graded}, if there 
for all $\si \in \G$ is an additive subgroup $R_{\si}$ of $R$ such that $R = \bigoplus_{\si \in \G} R_{\si}$
and for all $\si,\ta \in \G$ the inclusion $R_{\si} R_{\ta} \subseteq R_{\si \ta}$ holds,
if $(\si,\ta) \in G_2$, and $R_{\si} R_{\ta} = \{ 0 \}$, otherwise.
There are  relevant classes of rings that can be graded by groupoids, such as matrix rings, 
crossed product algebras defined by separable extensions and partial skew groupoid rings that are not, in a
natural way, graded by groups (see e.g. \cite{lundstrom2004,lundstrom2006,NyOP22018}). 

Let $L/K$ be a finite separable (not necessarily normal) field extension.
The corresponding crossed product construction from \cite{lundstrom2005} runs as follows.
Let $N$ be a normal closure of $L/K$ and let $G$ denote the Galois group of $N/K$. 
Let $L=L_1 , L_2 \ldots , L_n$ denote the different conjugate fields of $L$ under the action of $G$. 
If $1 \leq i,j \leq n$, then let $G_{ij}$ denote the set of
field isomorphisms from $L_j$ to $L_i$. If $\si \in G_{ij}$,
then we indicate this by writing $d(\si) = j$ and $r(\si) = i$. 
If we let $\G$ denote the union of the $G_{ij}$, $1 \leq i,j \leq n$, 
then $\G$ is a groupoid. 
The crossed product algebra $(L/K,\beta)$
is defined as the additive group 
$\bigoplus_{\si \in \G} L_{r(\si)} u_\si$ with multiplication defined
by the $K$-linear extension of (\ref{multiplication}), if
$(\si,\ta) \in \G_2$, and $x u_\si y u_\ta = 0$, otherwise, for all $\si,\ta \in \G$ 
and all $x \in L_{d(\si)}$, $y \in L_{r(\ta)}$, where $\beta$ is a cocycle on the groupoid $\G$. 
The latter means that (see e.g. \cite{ren} for the details) $\beta_{\si,\ta}$ is defined precisely when
$(\si,\ta) \in \G_2$ and that it then satisfies $\beta_{\si,\ta} \in L_{r(\si)} \setminus \{ 0 \}$ 
and (\ref{cocycle}) for all $(\si,\ta,\rho) \in \G_3$. 
We also assume that $\beta$ satisfies (\ref{identity}) whenever $\si$ or $\ta$ is an identity map on some of the
conjugate fields of $L$. 
If we put $R = (L/K,\beta)$
and $R_\si = L u_\si$, for $\si \in \G$, then,
clearly, $R$ is graded by the groupoid $\G$. In fact, it is even {\it strongly graded} in the sense
that $R_\si R_\ta = R_{\si\ta}$, for $(\si,\ta) \in \G_2$.
Furthermore, $R$ is unital with $1_R = \sum_{i=1}^n 1_{L_i} u_{{\rm id}_{L_i}}$.
Note that if $L/K$ is Galois, then $\G = G$ and thus $(L/K,\beta)$
coincides with the classical crossed product algebra construction above.

Turning back to our question, suppose now that $L/K$
is a separable field extension of {\it infinite} degree.
It is not hard to imagine that if we mimic the construction from the finite case, 
then we would still get a groupoid graded ring $R = (L/K,\beta)$
which, however, now looses the property of being unital (see Example \ref{crossedproductseparable} for the details).
Nevertheless, $R$ always retains a weaker form of unitality as a graded ring, in the following sense
(introduced in \cite{oinert2012}).
If $R$ is a ring graded by a groupoid $\G$, then $R$ is said to be {\it object unital} if 
for all $e \in \G_0$ the ring $R_e$ is unital and for all $\si \in \G$ and all 
$r \in R_{\si}$ the equalities $1_{R_{r(\si)}} r = r 1_{R_{d(\si)}} = r$ hold.
This is the class of objects that we wish to study.
Note that in \cite{oinert2012} the concept object unital is called ``locally unital''. 
However, since the latter may be confused with a different concept in ring theory
(see Section \ref{section:gradedrings}) we have chosen to rename it here. 
It is also worth mentioning that crossed products induced by group actions in rings with local units  
has already been used in \cite{ELGO}  in order to obtain an analogue of the 
Chase–Harrison–Rosenberg seven terms exact sequence for this family of rings.

Here is a detailed outline of the article.

In Section \ref{section:gradedrings}, 
we state our conventions concerning rings and we
generalize some results for unital group (and groupoid) graded rings
to a non-unital situation (see Proposition \ref{objectunitalnonzero},
Proposition \ref{propunital} and Proposition \ref{propstronglygraded}).

In Section \ref{section:objectcrossedproducts}, 
we introduce object crossed products (see Definition \ref{defcrossedproduct}).
We show that these structures can be parametrized by object crossed systems (see Definition \ref{defcrossedsystem}
and Proposition \ref{propringtocrossed}).
We specialize this construction to the case of separable (but not necessarily finite or normal) field extensions
and prove that the resulting ring is simple (see Proposition \ref{crossedproductseparable}).
After that, we obtain criteria for when object crossed products are object skew groupoid rings
and object twisted groupoid rings (see Proposition \ref{propskew} and Proposition \ref{proptwisted}),
thereby generalizing results for group (and groupoid) graded unital rings to a non-unital situation. 

In Section \ref{section:separability}, we obtain a result (see Proposition \ref{propsepidempotent})
on the separability of extensions of rings with enough idempotents which, in turn, 
is used (through a series of propositions) to obtain the following non-unital 
separability result for object unital groupoid graded rings in connection to 
properties of the groups $\G(e) = \{ \si \in \G \mid d(\si)=r(\si)=e \}$ and
local traces ${\rm tr}_e$ defined for $e \in \G_0$ (see Definition \ref{definition:trace}).

\begin{thm}\label{maintheorem1}
If $\G$ is a groupoid and $R$ is a ring which is object unital strongly $\G$-graded, then 
the ring extension $R/R_0$ is separable, if and only if, 
for all $e \in \G_0$, the group $\G(e)$ is finite and $1_{R_e} \in {\rm tr}_e(Z(R_e))$.
\end{thm}

Here we let $Z(R_e)$ denote the center of $R_e$, that is the set of all elements of $R_e$
that commute all other elements of $R_e,$ and we put $R_0=\oplus_{e\in \G_0}R_e$.
In the same section, we specialize Theorem \ref{maintheorem1} to object crossed products in the following result.

\begin{thm}\label{maintheorem2}
If $\G$ is a groupoid and $R = A \rtimes^{\alpha}_{\beta} \G$ is an object crossed product, 
then the ring extension $R/R_0$ is separable, if and only if,
for all $e \in \G_0$, the group $\G(e)$ is finite and
there exists $a \in A_e$ such that $\sum_{\si \in \G(e)} \alpha_\si(a) = 1_{A_e}$.
\end{thm}

This is, in turn, specialized to obtain separability results for object skew group rings,
object twisted group rings and groupoid rings, in particular infinite matrix rings
(see Proposition \ref{separablecrossed}, Proposition \ref{separabletwisted} and Proposition \ref{separablematrices}). 
The paper ends with the following separability result for object crossed products
(see Example \ref{crossedproductseparable} and Proposition \ref{maincrossed}).

\begin{thm}\label{maintheorem3}
Suppose that $L/K$ is a separable (not necessarily finite) field extension. Then the object crossed
product $R=(L/K,\beta)$ is separable over $R_0$, if and only if, ${\rm Aut}_K(L)$ is finite. 
In particular, if $L/K$ is Galois, then $R$ is separable over $R_0 = L$, if and only if, 
$L/K$ is finite.
\end{thm}

To the knowledge of the authors of the present article,
there are very few structural results concerning separability
of {\it non-unital} and {\it noncommutative} ring extensions in the literature.
In fact, the only article addressing this that we have found is the paper \cite{brzezinski2005}
by Brzezinski, Kadison and Wisbauer.
However, whereas the results in \cite{brzezinski2005} discuss theoretical connections between 
separablility of $A$-rings and $A$-corings (see \cite[Theorem 2.6]{brzezinski2005}), 
our results can be used to easily construct large families of 
concrete examples of non-unital separable ring extensions.
Note that a more restricted form of separability of non-unital algebras, defined over unital commutative rings,
has been considered by Taylor \cite{taylor1982}, while, 
in the context of co-algebras, the study of the splitting prperty for the  
comultiplication  map have been treated in \cite{B-GT}.
For a thorough account of the historical development of the notion of separability,
see the paper \cite{wisbauer2016} by Wisbauer.

\section{Preliminaries}\label{section:gradedrings}

In this section, we state our conventions concerning rings and modules. Thereafter, we show
some generalizations of classical results for unital group (and groupoid) graded rings
to a non-unital situation (see Proposition \ref{objectunitalnonzero},
Proposition \ref{propunital} and Proposition \ref{propstronglygraded}).
Note that parts of these results have previously been obtained in 
\cite[Proposition 3.1]{oinert2012}. We have, nonetheless, for the convenience 
of the reader, chosen to include them with complete (but slightly different) proofs.

\subsection*{Rings}
Throughout this section, $R$ denotes an associative but not necessarily unital ring.
If $X,Y \subseteq R$, then we let $XY$ denote
the set of finite sums of elements of the form $xy$ for $x \in X$ and $y \in Y$. 
The ring $R$ is called {\it idem\-potent} if $R R = R$.
Following Fuller in \cite{Fu} we say that $R$ has {\it enough idempotents} if there exists a set 
$\{ e_i \}_{i \in I}$ of orthogonal idempotents in $R$ (called a complete set of idempotents for $R$) 
such that $R = \bigoplus_{i\in I} R e_i = \bigoplus_{i \in I} e_i R.$
Following \'{A}nh and M\'{a}rki in \cite{AnM} we say that $R$ is {\it locally unital} if for all $n \in \N$ and all
$r_1,\ldots , r_n \in R$ there is an idempotent $e \in R$ such that for all $i \in \{1, \ldots , n\}$
the equalities $e r_i = r_i e = r_i$ hold.
The ring $R$ is called {\it s-unital} if for all $r \in R$ the relation $r \in Rr \cap rR$ holds.
The following chain of implications hold (see e.g. \cite{Ny2019}) for all rings
\begin{equation}\label{chain}
\begin{array}{rcl}
\mbox{unital} & \Rightarrow & \mbox{enough idempotents} \\
              & \Rightarrow & \mbox{locally unital} \\
              & \Rightarrow & \mbox{s-unital} \\
              & \Rightarrow & \mbox{idempotent.}
\end{array}
\end{equation}
We let RING (URING) denote the category having (unital) rings as objects 
and ring homomorphisms (that respect multiplicative identities) as morphisms.
If $R$ is a unital ring, then we let $U(R)$ denote the multiplicative group of invertible elements of $R$.

\subsection*{Modules}

Let $A$ be a ring. By a left $A$-module we mean an additive group $M$
equipped with a biadditive map $A \times M \ni (a,m) \mapsto am \in M$ satisfying
$a_1 ( a_2 m ) = (a_1 a_2) m$ for $a_1,a_2 \in A$ and $m \in M$.
Unless otherwise stated, we always assume that modules are {\it unitary} in the sense that $AM=M$.
Analogously, right modules are defined. Let $A'$ be another ring. If $M$ is both a left $A$-module and a right $A'$-module, 
then we say that $M$ is an $A$-$A'$-bimodule if $(am)a' = a(ma')$ for $a \in A$, $a' \in A'$ and $m \in M$.  
If $M$ and $N$ are left $A$-modules, then we let ${\rm Hom}_A(M,N)$ denote the additive group 
of left $A$-module homomorphisms.
Let $A/B$ be a ring extension. By this we mean that $B$ is a ring with $A \supseteq B$.
The ring $A$ is considered a $B$-bimodule in the natural sense.
We let $\varphi : B \rightarrow A$ denote the inclusion map. To $\varphi$ we associate the following functors.
The {\it restriction} functor $\varphi_* : A\mbox{-mod} \to B\mbox{-mod}$ which to a left $A$-module
associates its natural structure as a left $B$-module.
The {\it induction} functor $\varphi^* : B\mbox{-mod} \to A\mbox{-mod}$ which to a left $B$-module $M$
associates the left $A$-module $A \otimes_B M$.
In Section \ref{section:separability}, we study separability of these functors 
in the context of groupoid graded rings.

\subsection*{Graded rings}
For the rest of this section, $\G$ denotes a groupoid and $R$ is a ring graded by $\G$.
We let $\G$-RING ($\G$-URING) denote the category having (unital) $\G$-graded rings as objects and graded ring homomorphisms
(that respect multiplicative identities) as morphisms.
The set $H(R)=\bigcup_{\si \in \G} R_\si$ is called the set of \emph{homogeneous elements of $R$}. 
If $r \in R_\si \setminus \{ 0 \}$, then we say that $r$ is of \emph {degree $\si$} and write $\deg(r)=\si$. 
Any $r \in R$ has a unique decomposition $r = \sum_{\si \in \G} r_{\si}$, where
$r_{\si} \in R_{\si}$, for $\si \in \G$, and all but finitely many of the $r_{\si}$ are zero.  
Recall from the introduction, that we say that $R$ is {\it object unital} if for all $e \in \G_0$ the ring $R_e$ is unital and 
for all $\si \in \G$ and all $r \in R_{\si}$ the equalities $1_{R_{r(\si)}} r = r 1_{R_{d(\si)}} = r$ hold.
Note that if $R$ is object unital, then $R$ has enough idempotents given by $1_{R_e}$, for $e \in \G_0$. 
In that case, from (\ref{chain}), it follows that $R$ is also locally unital, s-unital and idempotent.
We let O-$\G$-RING denote the category having object unital $\G$-graded rings as objects and as morphisms
graded ring homomorphisms $f : R \to S$ satisfying $f( 1_{R_e} ) = 1_{S_e}$ for $e \in \G_0$.
Note also that if $\G$ is a group, then
\[
\mbox{O-$\G$-RING = $\G$-URING}
\]
and if $\G$ is the trivial group, then 
\[
\mbox{URING = O-$\G$-RING = $\G$-URING.}
\]
To state the first result, we need some notions concerning groupoids.
Recall that a non-empty subset $\mathcal{H}$ of $\G$ is called a subgroupoid of $\G$ if
(i) $h \in \mathcal{H}$ $\Rightarrow$ $h^{-1} \in \mathcal{H}$, and
(ii) $h_1,h_2 \in \mathcal{H} \cap \G_2$ $\Rightarrow$ $h_1 h_2 \in \mathcal{H}$.
In that case, $\mathcal{H}$ is called a wide subgroupoid of $\G$ if $\mathcal{H}_0 = \G_0$.

The following result generalizes \cite[Proposition 2.1.1]{lundstrom2004} from the unital case
to the object unital situation.

\begin{prop}\label{objectunitalnonzero}
If $R$ is object unital and we put 
\[
\G' = \{ \si \in G \mid 1_{R_{r(\si)}} \neq 0 \ \mbox{and} \ 1_{R_{d(\si)}} \neq 0 \},
\] 
then $\G'$ is a subgroupoid of $\G$ and $R = \oplus_{\si \in \G'} R_{\si}$.
The subgroupoid $\G'$ is wide if and only if for all $e \in \G_0$ the element $1_e$ is non-zero. 
\end{prop}

\begin{proof}
The first statement follows immediately from the fact that
if $( \si,\ta ) \in \G_2$, then $d( \si^{-1} ) = r( \si )$, $r( \si^{-1} ) = d( \si )$,
$r( \si \ta ) = r( \si )$ and $d( \si \ta ) = d( \ta )$.
The equality $R = \oplus_{\si \in \G'} R_{\si}$ follows from the following chain of equalities
\[
\begin{array}{rcl}
R & = & \oplus_{\si \in \G} R_\si \\
  & = & ( \oplus_{\si \in \G'} R_\si) \oplus ( \oplus_{\si \in \G \setminus \G'} R_\si ) \\
  & = & ( \oplus_{\si \in \G'} R_\si) \oplus ( \oplus_{\si \in \G \setminus \G'} 1_{R_{r(\si)}} R_\si 1_{R_{d(\si)}}  ) \\
  & = & \oplus_{\si \in \G'} R_{\si}.
\end{array}
\]
The last part follows immediately.
\end{proof}

In light of Proposition \ref{objectunitalnonzero}, we will from now on make the following 
\begin{ass} 
If $R$ is object unital, then for all $e \in \G_0$, $1_{R_e} \neq 0$.
\end{ass}
To state the next two results, we need some more notions.
If $X$ is a subset of $\G$, then we put $R_X = \oplus_{x \in X} R_x$.
Suppose that $\mathcal{H}$ is a subgroupoid of $\G$. Then $R_{\mathcal{H}}$ is a subring of $R$.
Indeed, if $R$ is an object in $\G$-RING ($\G$-URING or O-$\G$-RING), then
$R_{\mathcal{H}}$ is an object in $\mathcal{H}$-RING ($\mathcal{H}$-URING or O-$\mathcal{H}$-RING). 
Suppose that $E$ is a subset of $\G_0$.
We put 
\[
\G(E) = \{ \si \in \G \mid \mbox{$d(\si) \in E$ and $r(\si) \in E$} \}.
\]
It is easy to see that $\G(E)$ is a subgroupoid of $\G$ which we call the {\em principal groupoid associated to} $E$. 
If $e \in E$, then $\G( \{ e \} )$ is a group which we call the {\em principal group associated} $e$
and we denote this group by $\G(e)$.

The next result generalizes \cite[Proposition 1.1.1]{N82} (and \cite[Proposition 2.1.1]{lundstrom2004})
from the unital group (and groupoid) graded case to a non-unital situation.

\begin{prop}\label{propunital}
The following are equivalent:
\begin{itemize}

\item[(i)] the ring $R$ is object unital; 

\item[(ii)] for all finite subsets $E$ of $\G_0$ the ring $R_{\G(E)}$ is unital.

\end{itemize}
In that case, if $E$ is a finite subset of $\G_0$, then $1_{R_{\G(E)}} = \sum_{e \in E} 1_{R_e}$.
\end{prop}

\begin{proof}
(i)$\Rightarrow$(ii): Take a finite subset $E$ of $\G_0$
and a homogeneous $r \in R_{\si}$ for some $\si \in \G(E)$.
Put $s = \sum_{e \in E} 1_{R_e}$. From object unitality of $R$
it follows that $sr = \sum_{e \in E} 1_{R_e} r = 1_{R_{r(\si)}} r = r$ and
$rs = \sum_{e \in E} r 1_{R_e} = r 1_{R_{d(\si)}} = r$. 
Therefore $R_{\G(E)}$ is unital and $1_{R_{\G(E)}} = s$.

(ii)$\Rightarrow$(i): 
Take $e \in E$. From (ii) we get that $R_{\G(e)}$ is unital. Since $\G(e)$ is a group
it follows from \cite[Prop. 1.1.1.1]{N82} that $R_e$ is unital and $1_{R_{\G(e)}} = 1_{R_e}$.
Take $\si \in \G$ and $r \in R_{\si}$. If $d(\si) = r(\si)=:e$, then $\si \in \G( e )$
so that $1_{R_{r(\si)}} r = 1_{R_e} r = r = 1_{R_e} = 1_{R_{d(\si)}}$.
Now suppose that $d(\si) \neq r(\si)$.
Put $e = d(\si)$, $f = r(\si)$ and $E = \{ e , f \}$.
From (ii) we get that $R_{\G(E)}$ is unital.
Suppose that $1_{R_{\G(E)}} = \sum_{\rho \in \G(E)} x_{\rho}$ for some $x_{\rho} \in R_{\rho}$
and $x_{\rho} = 0$ for all but finitely many $\rho \in \G(E)$.
Take $\epsilon \in \G(E)$. Then 
$x_{\epsilon} = 1_{R_{\G(E)}} x_{\epsilon} = \sum_{\rho \in \G(E)} x_{\rho} x_{\epsilon}$
from which it follows that $x_{\rho} x_{\epsilon} = 0$ if $\rho \in \G(E)$ and $\epsilon \in \G(E) \setminus E$.
Thus if $\epsilon \in \G(E) \setminus E$, then 
$x_{\rho} 1_{R_{\G(E)}} = \sum_{\epsilon \in \G(E)} x_{\rho} x_{\epsilon} = 0$.
Hence $1_{R_{\G(E)}} = x_e + x_f$. But $1_{R_e} = 1_{R_e} 1_{R_{\G(E)}} = 1_{R_e} x_e + 1_{R_e} x_f = x_e$
and analogously $1_{R_f} = x_f$ so we get that $1_{R_{\G(E)}} = 1_{R_e} + 1_{R_f}$.
Thus, finally, we get that 
$1_{R_f} r = 1_{R_f} r + 1_{R_e} r = 1_{R_{\G(E)}} r = r$ and analogously that $r 1_{R_e} = r$.
This proves that $R$ is object unital.

The last part follows from the proof of (i)$\Rightarrow$(ii).
\end{proof}

\begin{rem} 
Notice that for every subgroupoid $\mathcal{H}$ of $\G$  the ring  $R_\mathcal{H}$ is $\G$-graded 
by setting $R_g=0$ if $g\notin \mathcal{H}.$ 
Then, it follows by (ii) of Proposition \ref{propunital} that an object unital graded ring is 
a direct limit of unital graded rings. Indeed,
let $\mathcal{F}$ be the family of finite subsets of $\G_0$.
Then, for all $E \in \mathcal{F}$, $R_{\G(E)}$ is a unital ring, and $\varinjlim_{E \in \mathcal{F}} R_{\G(E)} = R.$
\end{rem}

The following result generalizes \cite[Lemma I.3.2]{N82} (and \cite[Lemma 3.2]{lundstrom2006})
from the unital group (and groupoid) graded case to a non-unital situation.

\begin{prop}\label{propstronglygraded}
If $R$ is object unital, then the following are equivalent: 
\begin{itemize}

\item[(i)] $R$ is strongly graded;
 
\item[(ii)] for all finite subsets $E$ of $\G_0$ the ring $R_{\G(E)}$ is strongly $\G(E)$-graded;
 
\item[(iii)] for all $\si \in \G$ the equality $R_{r(\si)} = R_\si R_{\si^{-1}}$ holds;

\item[(iv)] for all $\si \in \G$ the relation $1_{R_{r(\si)}} \in R_\si R_{\si^{-1}}$ holds.

\end{itemize}
\end{prop}

\begin{proof}
The implications (i)$\Rightarrow$(ii)$\Rightarrow$(iii)$\Rightarrow$(iv) are clear. 
Now we show (iv)$\Rightarrow$(i). To this end, take $(\si,\ta) \in \G_2$. 
Since $r(\si) = r(\si\ta)$ it follows that $1_{ R_{r(\si)} } = 1_{ R_{ r(\si\ta) } }$.
Thus, from object unitality of $R$, we get that
$R_\si R_\ta \subseteq R_{\si \ta} = 1_{R_{r(\si\ta)}} R_{\si\ta} = 1_{R_{r(\si)}} R_{\si\ta} \subseteq 
R_\si R_{\si^{-1}} R_{\si\ta} \subseteq R_{\si} R_{\si^{-1} \si \ta} = R_\si R_\ta.$
\end{proof}

\section{Object crossed products}\label{section:objectcrossedproducts}

In this section, we introduce object crossed products (see Definition \ref{defcrossedproduct}).
We show that these structures can be parametrized by object crossed systems (see Definition \ref{defcrossedsystem}
and Proposition \ref{propringtocrossed}).
We specialize this construction to the case of separable (but not necessarily finite or normal) field extensions
and prove that the resulting ring is simple (see Proposition \ref{crossedproductseparable}).
After that, we obtain criteria for when object crossed products are object skew groupoid rings
and object twisted groupoid rings (see Proposition \ref{propskew} and Proposition \ref{proptwisted}),
thereby generalizing results for group (and groupoid) graded unital rings to a non-unital situation. 
Throughout this section, $\G$ denotes a groupoid and $R$ denotes a $\G$-graded ring which is object unital.

\begin{defi}
We say that a homogeneous element $r \in R_\si$ is {\it object invertible} if there exists 
$s \in R_{\si^{-1}}$ such that $rs = 1_{R_{r(\si)}}$ and $sr = 1_{R_{d(\si)}}$.
We will refer to $s$ as the {\it object inverse} of $r$.
\end{defi}

The usage of the term ``the object inverse'' is justified by the next result.

\begin{prop}
Object inverses are unique.
\end{prop}

\begin{proof}
Suppose that $\si \in \G$, $r \in R_\si$ and $s,t \in R_{\si^{-1}}$ satisfy the equalities 
$rs = rt = 1_{R_{r(\si)}}$ and $tr = sr = 1_{R_{d(\si)}}$. 
Then $s = s 1_{R_{d(\si^{-1})}} =
s 1_{R_{r(\si)}} =
s r t =
1_{R_{d(\si)}} t =
1_{R_{c(\si^{-1})}} t = t$.
\end{proof}

\begin{defi}
We let $U^{\rm gr}(R)$ denote the set of all object invertible elements of $R$.
We define a groupoid structure on $U^{\rm gr}(R)$ in the following way.
Take $r,s \in U^{\rm gr}(R)$. The groupoid composition of $r$ and $s$ is defined and equal to $rs$
precisely when $d( {\rm deg}(r) ) = c( {\rm deg}(s) )$. 
The groupoid inverse of $r$ is defined to be the object inverse of $r$.
We put $R_0 = \oplus_{e \in \G_0} R_e$ and we
consider $R_0$ as a $\G$-graded ring in the following sense. 
If $\si \in \G$, then $(R_0)_\si = R_\si$, if $\si \in \G_0$, and $(R_0)_\si = \{ 0 \}$, otherwise.
Then $U^{\rm gr}(R_0)$ equals the direct sum groupoid $\biguplus_{e \in \G_0} U(R_e)$
of the groups $\{ U(R_e) \}_{e \in \G_0}$.
\end{defi}

To state the next proposition, we need some notions concerning groupoids.
Suppose that $\mathcal{H}$ is a groupoid and let $f : \G \to \mathcal{H}$ be a map.
The {\it image} and {\it kernel} of $f$ are defined by respectively
$
{\rm Im}(f) = \{ f(g) \mid g \in \G \}
$
and
$
{\rm Ker}(f) = \{ g \in \G \mid f(g) \in \mathcal{H}_0 \}.
$
The map $f$ is said to be a groupoid homomorphism if for all $g,h \in \G$
with $gh$ defined, $f(g) f(h)$ is also defined and $f(g h) = f(g) f(h)$.
In that case, $f$ is said to be strong if for all $g,h \in \G$
with $f(g)f(h)$ defined, $gh$ is also defined.
It is clear that if $f$ is a (strong) groupoid homomorphism, 
then ${\rm Ker}(f)$ (${\rm Im}(f)$) is a subgroupoid of $\G$ ($\mathcal{H}$).

\begin{prop}\label{kernel}
The degree map restricts to a groupoid homomorphism 
\[
{\rm deg} : U^{\rm gr}(R) \to \G
\]
with kernel equal to $U^{\rm gr}(R_0)$
\end{prop}

\begin{proof}
This is clear.
\end{proof}

\begin{defi}\label{defcrossedproduct}
We say that $R$ is an {\it object crossed product}
if for all $\si \in \G$ the relation $U^{\rm gr}(R) \cap R_\si \neq \emptyset$ holds.
Note that this condition is equivalent to saying that the groupoid homomorphism
${\rm deg} : U^{\rm gr}(R) \to \G$ from Proposition \ref{kernel} is surjective.
Proposition \ref{propstronglygraded}(iv) implies that all object crossed products are strongly graded.
\end{defi}

\begin{defi}\label{defcrossedsystem}
Suppose that we are given a collection $A = ( A_e )_{e \in \G_0}$ of non-zero unital rings $A_e$
and put $A_0 = \oplus_{e \in \G_0} A_e$.
For all $e,f \in \G_0$ let $\iso_{e,f}(A)$ denote the set of ring isomorphisms $A_f \to A_e$ (respecting identity elements).
We let $\iso(A)$ denote the disjoint union $\biguplus_{e,f \in \G_0} \iso_{e,f}(A)$ and
we define a groupoid structure on $\iso(A)$ in the following way.
Take $e,e',f,f' \in \G_0$. The partial composition on $\iso(A)$
is defined to be the usual composition of functions
$\iso_{e,f}(A) \times \iso_{e',f'}(A) \to \iso_{e,f'}(A)$, when $f = e'$, and otherwise undefined.
By an {\it object crossed system} we mean a quadruple $( A , \G , \alpha , \beta )$ where
$\alpha : \G \to \iso(A)$ and $\beta : \G_2 \to U^{\rm gr}(A_0)$ are maps satisfying the 
following axioms for all $(\si,\ta,\rho) \in \G_3$ and all $a \in A_{d(\ta)}$
\begin{itemize}

\item[(i)] $\alpha_{\si} : A_{d(\si)} \to A_{r(\si)}$ and $\alpha_e = {\rm id}_{A_e}$ for $e \in \G_0$;

\item[(ii)] $\beta_{\si , \ta} \in U( A_{r(\si)} )$ and $\beta_{\si , d(\si)} = \beta_{ r(\si) , \si } = 1_{A_{r(\si)}}$;

\item[(iii)] $\alpha_{\si} ( \alpha_{\ta} (a) ) = \beta_{\si,\ta} \alpha_{\si \ta} (a) \beta_{\si,\ta}^{-1}$;

\item[(iv)] $\beta_{\si,\ta} \beta_{\si\ta,\rho} = \alpha_{\si}( \beta_{\ta,\rho} ) \beta_{\si,\ta\rho}$.

\end{itemize}
The map $\alpha$ is called a {\it weak action} of $\G$ on $A$ and $\beta$ is called an $\alpha$-{\it cocycle}.
\end{defi}

\begin{defi}\label{defobjectcrossedsystem}
Suppose that $( A , \G , \alpha , \beta )$ is an object crossed system.
Let $\{ u_{\si} \}_{\si \in \G}$ be a copy of $\G$.
We let $A \rtimes^{\alpha}_{\beta} \G$ denote the set of formal sums of the form $\sum_{\si \in \G} a_\si u_\si$
where $a_{\si} \in A_{r(\si)}$, for $\si \in \G$, and $a_{\si} = 0$, for all but finitely many $\si \in \G$.
If $\sum_{\si \in \G} a_\si u_\si$ and $\sum_{\si \in \G} a_\si' u_\si$ are two such formal sums, then their sum is defined to be
\[
\sum_{\si \in \G} a_\si u_\si + \sum_{\si \in \G} a_\si' u_\si = \sum_{\si \in \G} (a_\si + a_\si') u_\si.
\]
The product of two such formal sums is defined to be the additive extension of the relations
\begin{equation}\label{productfirst}
a_\si u_\si \cdot a_{\ta}' u_{\ta} = a_\si \alpha_\si ( a_{\ta} ' ) \beta_{\si,\ta} u_{\si \ta},
\end{equation}
when $(\si,\ta) \in \G_2$, and
\begin{equation}\label{productsecond}
a_\si u_\si \cdot a_{\ta}' u_{\ta} = 0,
\end{equation}
otherwise. For all $\si \in \G$ we put $( A \rtimes^{\alpha}_{\beta} \G )_\si = A_{r(\si)} u_\si$.
It is clear from (\ref{productfirst}) and (\ref{productsecond}) that this defines a $\G$-grading on $A \rtimes^{\alpha}_{\beta} \G$.
\end{defi}

\begin{rem}
The construction in Definition \ref{defobjectcrossedsystem} has previously appeared elsewhere.
Indeed, in \cite{oinert2010} the notion of a {\it Category crossed product} was defined.
\end{rem}

The next result generalizes \cite[Propositions 1.4.1-1.4.2]{nastasescu2004}
from the unital group graded case to a non-unital groupoid graded situation.

\begin{prop}\label{propringtocrossed}
If $( A , \G , \alpha , \beta )$ is an object crossed system, then $A \rtimes^{\alpha}_{\beta} \G$ is an 
object unital groupoid graded ring which is an object crossed product.
Conversely, any object crossed product can be presented in this way.
\end{prop}

\begin{proof}
Suppose that $( A , \G , \alpha , \beta )$ is an object crossed system and put $R = A \rtimes^{\alpha}_{\beta} \G$.
We show that $R$ is an object crossed product.
It is clear from Definition \ref{defcrossedsystem}(i)(ii) that each $R_e$, for $e \in \G_0$, is unital
with $1_{R_e} = 1_{A_e} u_e$ and that for all $\si \in \G$ and all $r \in R_{\si}$ 
the equalities $1_{R_{r(\si)}} r = r 1_{R_{d(\si)}} = r$ hold.
Next we show that $R$ is associative. To this end, take $(\si,\ta,\rho) \in \G_3$,
$a \in A_{r(\si)}$, $b \in A_{r(\ta)} = A_{d(\si)}$ and $c \in A_{r(\rho)} = A_{d(\ta)}$.
Then, from Definition \ref{defcrossedsystem}(iii,iv), it follows that
\[
\begin{array}{rcl}

a u_\si ( b u_\ta c u_{\rho} ) & = & a u_\si (b \alpha_\ta ( c ) \beta_{\ta,\rho} u_{\ta \rho} ) \\
                               & = & a \alpha_\si ( b \alpha_\ta ( c ) \beta_{\ta,\rho} ) \beta_{\si,\ta\rho} u_{\si\ta\rho} \\
                               & = & a \alpha_\si(b) \alpha_\si(\alpha_\ta(c)) \alpha_\si(\beta_{\ta,\rho}) \beta_{\si,\ta\rho} u_{\si\ta\rho} \\
                               & = & a \alpha_\si(b) \beta_{\si,\ta} \alpha_{\si\ta}(c) \beta_{\si,\ta}^{-1}
                               \beta_{\si,\ta} \beta_{\si\ta,\rho} u_{\si\ta\rho} \\
                               & = & a \alpha_\si (b) \beta_{\si,\ta} \alpha_{\si\ta}(c) \beta_{\si\ta,\rho} u_{\si\ta\rho} \\
                               & = & ( a \alpha_\si (b) \beta_{\si,\ta} u_{\si\ta} ) c u_{\rho} \\                           
                               & = & ( a u_\si b u_\ta ) c u_{\rho}.
\end{array}
\]
Now we show that $R$ is a crossed product. Take $\si \in \G$.
Put $r = \beta_{ \si,\si^{-1} }^{-1} u_\si$ and $s = 1_{d(\si)} u_{\si^{-1}}$. Then
\[
\begin{array}{rcl}
rs & = & ( \beta_{ \si,\si^{-1} }^{-1} u_\si ) ( 1_{d(\si)} u_{\si^{-1}} ) \\
   & = & \beta_{ \si,\si^{-1} }^{-1} \beta_{ \si,\si^{-1} } u_{r(\si)} \\
   & = & 1_{A_{r(\si)}} u_{r(\si)}
\end{array}
\]
and, from Definition \ref{defcrossedsystem}(iv), it follows that
\[
\begin{array}{rcl}
sr & = & ( 1_{d(\si)} u_{\si^{-1}} ) ( \beta_{ \si,\si^{-1} }^{-1} u_\si )  \\
   & = &  \alpha_{\si^{-1}}( \beta_{ \si,\si^{-1} }^{-1} ) \beta_{ \si^{-1},\si } u_{d(\si)} \\
   & = & \beta_{\si^{-1},\si}^{-1} \beta_{ \si^{-1},\si } u_{d(\si)} \\
   & = & 1_{A_{d(\si)}} u_{d(\si)}.
\end{array}
\]
Conversely, suppose that $R$ is an object unital ring graded by $\G$ which is a crossed product.
For all $e \in \G_0$ put $A_e = R_e$.
For all $\si \in \G$ choose an element $u_\si \in R_\si \cap U^{\rm gr}(R)$.
Since $R$ is object unital we may choose $u_e = 1_{R_e}$ for $e \in \G_0$.
Define maps $\alpha : \G \to \iso(A)$ and $\beta : \G_2 \to U^{\rm gr}(A_0)$ in the following way.
Take $(\si,\ta) \in \G_2$ and $a \in A_{d(\tau)}$. Let $v_{\ta^{-1}}$ denote the object inverse of $u_\ta$. Put
$\alpha_{\ta}(a) = u_{\ta} a v_{\ta^{-1}}$
and
$\beta_{\si,\ta} = u_\si u_\ta v_{(\si\ta)^{-1}}.$
These maps are well defined since ${\rm deg}(  u_{\ta} a v_{\ta^{-1}} ) = \ta \ta^{-1} = r(\ta)$,
${\rm deg}(  u_\si u_\ta v_{(\si\ta)^{-1}} ) = \si \ta (\si \ta)^{-1} = r(\si)$ and
$u_\si u_\ta v_{(\si\ta)^{-1}}$ is invertible in the ring $A_{r(\si)}$ with multiplicative inverse given by $u_{\si\ta} v_{\ta^{-1}} v_{\si^{-1}}$. 
Now we show conditions (i)-(iv) of Definition \ref{defcrossedsystem}.
Conditions (i) and (ii) hold since $u_e = v_e = 1_{R_e}$ for $e \in \G_0$.
Now we show condition (iii):
\[
\begin{array}{rcl}
\alpha_{\si} ( \alpha_{\ta} (a) ) &=& u_\si u_\ta a v_{\ta^{-1}} v_{\si^{-1}} \\
                                  &=& u_\si u_\ta v_{(\si\ta)^{-1}} u_{\si\ta} a v_{(\si\ta)^{-1}} u_{\si\ta} v_{\ta^{-1}} v_{\si^{-1}} \\
                                  &=& \beta_{\si,\ta} \alpha_{\si \ta} (a) \beta_{\si,\ta}^{-1}.
\end{array}
\]
Next we show condition (iv). To this end, suppose that $(\si,\ta,\rho) \in \G_3$. Then
\[
\begin{array}{rcl}
\beta_{\si,\ta} \beta_{\si\ta,\rho} &=& u_\si u_\ta v_{(\si\ta)^{-1}} u_{\si\ta} u_{\rho} v_{(\si\ta\rho)^{-1}} \\
                                    &=& u_{\si} u_\ta u_{\rho} v_{(\ta\rho)^{-1}} v_{\si^{-1}} u_\si u_{\ta\rho} v_{(\si\ta\rho)^{-1}}  \\
                                    &=& \alpha_{\si}( \beta_{\ta,\rho} ) \beta_{\si,\ta\rho}.
\end{array}
\]
Now we show that $R$ is isomorphic as a ring to $A \rtimes^{\alpha}_{\beta} \G$.
Take $r_\si \in R_\si$. Since $R$ is object unital we get that 
$r_\si = r_\si 1_{R_{d(\si)}} = r_\si v_{\si^{-1}} u_\si \in A_{r(\si)} u_\si$.
Suppose that $a u_\si = b u_\si$ for some $a,b \in A_{r(\si)}$.
Then $a = a 1_{R_{r(\si)}} = a u_\si v_{\si^{-1}} = b u_\si v_{\si^{-1}} = b 1_{R_{r(\si)}} = b$.
This means that $u_\si$ spans $R_\si$ freely as a left $A_{r(\si)}$-module.
Define $\gamma : R \to A \rtimes^{\alpha}_{\beta} \G$ by the additive extension 
of the relations $\gamma(r_\si) = ( r_\si v_{\si^{-1}} ) u_\si$, for $\si \in \G$ and $r \in R_\si$,
and define $\delta :  A \rtimes^{\alpha}_{\beta} \G \to R$ by the additive extension
of the relations $\delta( a_{r(\si)} u_\si ) = a_{r(\si)} u_\si$,for $\si \in \G$ and $a \in R_{r(\si)}$.
It is clear that $\gamma \circ \delta = {\rm id}_{A \rtimes^{\alpha}_{\beta} \G}$ and $\delta \circ \gamma = {\rm id}_R$.  
To finish the proof we need to show that $\gamma$ respects multiplication.
To this end, take $(\si,\ta) \in \G_2$, $a \in A_{r(\si)}$ and $b \in A_{r(\ta)}=A_{d(\si)}$. Then
\[
\begin{array}{rcl}
\gamma( a u_\si b u_\ta ) &=& \gamma( a u_\si b v_{\si^{-1}} u_\si u_\ta v_{(\si\ta)^{-1}}) u_{\si\ta} \\
                          &=& \gamma( a \alpha_\si( b ) \beta_{\si,\ta} u_{\si\ta} ) \\
                          &=& a \alpha_\si( b ) \beta_{\si,\ta} u_{\si\ta} \\
                          &=& ( a u_\si )( b u_\ta ) \\
                          &=& \gamma( a u_\si ) \gamma( b u_\ta ).
\end{array}
\]
\end{proof}

\begin{ex}\label{crossedproductseparable}
Let $L/K$ be a separable field extension of possibly infinite degree.
Let $\overline{K}$ denote a fixed algebraic closure of $K$ containing $L$. 
Choose a normal closure $N/K$ of $L/K$ in $\overline{K}$.
Let $G$ denote the Galois group of $N/K$ and let $\{ L_i \}_{i \in I}$ be
the different conjugate fields of $L$ under the action of $G$.
If $i,j \in I$, then put $G_{ij} = \{ g|_{L_j} \mid g \in G \ \mbox{and} \ g(L_j)=L_i \}$.
If we put $\G = \biguplus_{i,j \in I} G_{ij}$, then $\G$ is, in a natural way, a groupoid with respect to composition.
Let $\tilde{L} = ( L_i )_{i \in I}$ and define $\alpha : \G \to \iso( \tilde{L} )$ by $\alpha_{\si} = \si$ for $\si \in G_{ij}$. 
If $\beta : \G_2 \to U^{\rm gr}(\tilde{L}_0)$ is a map satisfying conditions (ii) and (iv) of Definition \ref{defcrossedsystem}
(condition (iii) is automatically satisfied), then we say that 
$\tilde{L} \rtimes^{\alpha}_{\beta} \G$ is {\it the object crossed product
defined by $L/K$ and $\beta$} and we let it be denoted by $(L/K,\beta)$. 
If $L/K$ is a finite Galois extension, then $(L/K,\beta)$ coincides with the classical 
construction of a crossed product relative the extension $L/K$ (see e.g. \cite[p. 242]{reiner1975} or Section \ref{section:introduction}
of the present article). 
\end{ex}

\begin{prop}
With the above notations, the following assertions hold. 
\begin{itemize}

\item[(a)] The $K$-algebra $(L/K,\beta)$ is simple.

\item[(b)] The $K$-algebra $(L/K,\beta)$ is unital if and only if $I$ is finite.
In that case, $Z((L/K,\beta))$ is a field which is isomorphic to $L^G$. 

\item[(c)] If $I$ is infinite, then $Z((L/K,\beta)) = \{ 0 \}$.

\end{itemize}
\end{prop}

\begin{proof}
(a) Take a non-zero ideal $J$ of $(L/K,\beta)$. 
We claim that there exists $j \in I$ such that $u_{{\rm id}_{L_j}} \in J$. 
If we assume that the claim holds, then it follows that $J = (L/K,\beta)$.
Indeed, since all the fields $\{ L_i \}_{i \in I}$ are conjugated under the action of $G$ it follows that $\G$ is connected.
Therefore, for any $i \in I$, there exists a field isomorphism $\sigma_i : L_j \to L_i$ (induced by the action of $G$).
Hence, 
$u_{{\rm id}_{L_i}} = \beta_{\sigma_i,\sigma_i^{-1}}^{-1} u_{\sigma_i} u_{{\rm id}_{L_j}} u_{\sigma_i^{-1}} \in J$.
Thus, all $u_{{\rm id}_{L_i}}$, for $i \in I$, belong to $J$ and therefore $J = (L/K,\beta)$.
Now we show the claim. Take a non-zero element $x = \sum_{i=1}^n a_i u_{\sigma_i}$ in $(L/K,\beta)$,
with all $a_i \in L_{r(\sigma_i)}$ non-zero, all $\sigma_i$, for $i=1,\ldots,n$, distinct,
such that $n$ is chosen as small as possible subject to the condition that $x \in J$. 
Seeking a contradiction, suppose that $n > 1$.
Since $u_{r(\si_1)} x u_{d(\si_1)}$ also is a non-zero element of $J$ it follows that 
we may assume that $d(\si_i)=d(\si_1)$ and $r(\si_i)=r(\si_1)$ for all $i = 1,\ldots,n$.
Since $\si_1 \neq \si_2$ we may choose $a \in L_{d(\si_1)}$ with $\si_1(a) \neq \si_2(a)$.  
Then $x - \sigma_1(a) x a u_{d(\si_1)}$ is a non-zero element of $J$ of shorter length,
which is a contradiction. Therefore, $n=1$ and we can write $x = a u_{\si}$ for some $\si \in \G$
and some non-zero $a \in A_{r(\si)}$. Then 
$1_{L_{r(\si)}} u_{r(\si)} = \beta_{\si,\si^{-1}}^{-1} a^{-1} x 1_{d(\si)} u_{\si^{-1}} \in J$
and we have shown the claim made in the beginning of the proof. 

(b) Suppose that $(L/K,\beta)$ is unital and let $y$ be a multiplicative identity of $(L/K,\beta)$.
Since $y u_{{\rm id}_{L_i}} = u_{{\rm id}_{L_i}} y = y$ for all $i \in I$ it is clear that $I$ must be finite.
From Proposition \ref{propunital} it follows that 
$\sum_{i \in I} 1_{L_i} u_{{\rm id}_{L_i}}$ is a multiplicative identity of $(L/K,\beta)$.
The statement concerning the center of $(L/K,\beta)$ follows from the the same argument used in
\cite[Section 4]{lundstrom2005}.

(c) This follows from the fact that any simple non-unital $K$-algebra has center equal to $\{ 0 \}$
(see e.g. \cite[Theorem 2.1]{schafer1966}).
\end{proof}

\begin{defi}
If $( A , \G , \alpha , \beta )$ is a crossed system with $\beta$ trivial, that is if
for all $(\si,\ta) \in \G$ the relation $\beta_{\si,\ta} = 1_{A_{r(\si)}}$ holds, 
then we say that the corresponding object crossed product $A \rtimes^{\alpha}_{\beta} \G$ is 
an {\it object skew groupoid ring} and we denote it by $A \rtimes^{\alpha} \G$.
\end{defi}

\begin{prop}\label{propskew}
$R$ is an object skew groupoid ring, if and only if,
the degree map $U^{\rm gr}(R) \to \G$ is surjective and splits as a homomorphism of groupoids.
\end{prop}

\begin{proof}
First we show the ``only if'' statement. Suppose that $R$ is the object skew groupoid ring $A \rtimes^{\alpha} \G$.
Define a homomorphism of groupoids $\varphi : \G \to U^{\rm gr}(R)$ by saying that
$\varphi( \si ) = u_\si$, for $\si \in \G$. If $(\si,\ta) \in \G_2$, then
\[
\varphi( \si \ta ) = u_{\si \ta} = u_\si u_\ta = \varphi( \si ) \varphi( \ta ).
\]
Clearly, ${\rm deg} \circ \varphi = {\rm id}_{\G}$.
Now we show the ``if'' statement. Suppose that we are given a homomorphism of groupoids $\varphi : \G \to U^{\rm gr}(R)$
such that ${\rm deg} \circ \varphi = {\rm id}_{\G}$. For all $\si \in \G$ let $u_\si$ denote the object 
invertible element $\varphi(\si)$. Since groupoid homomorphisms preserve domain and range it is clear
that $\varphi(e) = 1_{R_e}$ for $e \in \G_0$. Define maps $\alpha : \G \to \iso(A)$ and $\beta : \G_2 \to U^{\rm gr}(A_0)$
using the elements $\{ u_\si \}_{\si \in \G}$ as in the proof of Proposition \ref{propringtocrossed}. 
Take $(\si,\ta) \in \G_2$. Let $v_{\ta^{-1}}$ denote the object inverse of $u_\ta$. Then
\[
\beta_{\si,\ta} = u_\si u_\ta v_{(\si\ta)^{-1}} 
 =  \varphi(\si) \varphi(\ta) \varphi( (\si \ta)^{-1} ) 
 =  \varphi( r( \si \ta ) ) 
 =  \varphi( r( \si ) )
 =  1_{R_{r(\si)}}.
\]
\end{proof}

\begin{rem}
Various properties of object skew groupoid rings have been studied in 
\cite{lundstrom2012} and \cite{nystedt2013} in the context of partially defined
groupoid dynamical systems on topological spaces.
\end{rem}

\begin{defi}
If $( A , \G , \alpha , \beta )$ is a crossed system with $\alpha$ trivial, that is if
all the rings $A_e$, for $e \in \G_0$, coincide with the same ring $B$, and for all $\si \in \G$,
the map $\alpha_\si : B \to B$ is the identity, then we say that the corresponding 
object crossed product $A \rtimes^{\alpha}_{\beta} \G$ is  
an {\it object twisted groupoid ring} and we denote it by $B \rtimes_{\beta} \G$.
In that case, $\beta_{\si,\ta} \in U(Z(B))$, for $(\si,\ta) \in \G_2$
(this follows from Definition \ref{defcrossedsystem}(iii)).
\end{defi}

\begin{prop}\label{proptwisted}
$R$ is an object twisted groupoid ring, if and only if, the following conditions hold:
\begin{itemize}

\item[(i)] all the rings $R_e$, for $e \in \G_0$, are copies of the same ring $B$
(the copy of an element $b \in B$ in $R_e$ is denoted by $b^e$);

\item[(ii)] for all $\si \in \G$ there is an object invertible element $u_\si \in R_\si$ with the property that
for all $b \in B$ the equality $b^{r(\si)} u_\si = u_\si b^{d(\si)}$ holds. 

\end{itemize}
\end{prop}

\begin{proof}
The ``only if'' statement is clear. Now we show the ``if'' statement. 
Take $\si \in \G$ and  an object invertible element $u_\si \in R_\si$ satisfying (ii). 
Let $v_{\si^{-1}}$ denote the object inverse of $u_\si$. Take $b^{ d(\si) } \in R_{d(\si)}$, for some $b \in B$.
Using the notation from Proposition \ref{propringtocrossed} we get that 
$\alpha_{\si}( b^{ d(\si) } ) = u_{\si} b^{ d(\si) } v_{\si^{-1}} = b^{r(\si)} u_\si v_{\si^{-1}} = b^{r(\si)}$ and
$\beta_{\si,\ta} = u_\si u_\ta v_{(\si\ta)^{-1}} \in U( R_{r(\si)} )$.
\end{proof}

\begin{rem}
If $( A , \G , \alpha , \beta )$ is a crossed system with $\alpha$ and $\beta$ trivial, that is if
all $(\si,\ta) \in \G$ the relation $\beta_{\si,\ta} = 1_{R_{r(\si)}}$ holds, and
all the rings $A_e$, for $e \in \G_0$, coincide with the same ring $B$, and for all $\si \in \G$,
the map $\alpha_\si : B \to B$ is the identity, then the corresponding 
object crossed product $A \rtimes^{\alpha}_{\beta} \G$ coincides with the 
{\it groupoid ring} of $\G$ over $B$ (see e.g. \cite[Example 2.1.2]{lundstrom2004}) and it is denoted by $B [ \G ]$.
\end{rem}

\begin{prop}
$R$ is a groupoid ring, if and only if, the following hold:
\begin{itemize}

\item[(i)] the degree map $U^{\rm gr}(R) \to \G$ is surjective and splits as a homomorphism of groupoids;

\item[(ii)] all the rings $R_e$, for $e \in \G_0$, are copies of the same ring $B$
(the copy of an element $b \in B$ in $R_e$ is denoted by $b^e$);

\item[(iii)] for all $\si \in \G$ there is an object invertible element $u_\si \in R_\si$ with the property that
for all $b \in B$ the equality $b^{r(\si)} u_\si = u_\si b^{d(\si)}$ holds. 

\end{itemize}
\end{prop}

\begin{proof}
This follows from Proposition \ref{propskew} and Proposition \ref{proptwisted}. 
\end{proof}

\section{Separability}\label{section:separability}

In this section, we first state the definition of separable functors and 
we obtain a criteria for separability for extensions of rings 
with enough idempotents (see Proposition \ref{propsepidempotent}).
After that, we use this result (and a series of propositions) to obtain a non-unital 
separability result (see Proposition \ref{separableprop}) for object unital groupoid graded rings.
We specialize this result to object crossed products, and then, in turn,
to object skew group rings, object twisted group rings and groupoid rings, and, in particular, infinite matrix rings
(see Proposition \ref{separablecrossed}, Proposition \ref{separabletwisted} and Proposition \ref{separablematrices}). 

\subsection*{Separable functors}

Let $C$ and $D$ be categories.
We let ${\rm Ob}(C)$ denote the collection of objects of $C$.
If $M,N \in {\rm Ob}(C)$, then we let ${\rm Hom}_C(M,N)$ denote the set of morphisms in $C$ from $M$ to $N$.
Following \cite{nastasescu1989}, we say that a covariant functor $F : C \to D$ is {\it separable}
if for all $M,N \in {\rm Ob}(C)$, there is a map 
$\varphi_{M,N}^F : {\rm Hom}_D(F(M),F(N)) \to {\rm Hom}_C(M,N)$
satisfying
\begin{itemize}

\item[(SF1)] $\varphi_{M,M'}^F( F(\alpha) ) = \alpha$, and

\item[(SF2)] $F(\beta) f = g F(\alpha)$ $\Rightarrow$ 
$\beta \varphi_{M,N}^F(f) = \varphi_{M',N'}^F(g) \alpha$,

\end{itemize}
for all $M,N,M',N' \in {\rm ob}(C)$,
all $\alpha \in {\rm Hom}_C(M,M')$, all $\beta \in {\rm Hom}_C(N,N')$,
all $f \in {\rm Hom}_D(F(M),F(N))$ and all $g \in {\rm Hom}_D(F(N'),F(M'))$.

\subsection*{Rings with enough idempotents}
Let $A/B$ be a ring extension and let $\varphi : B \to A$ denote the inclusion map.
Recall from Section \ref{section:gradedrings} that we then can define the associated 
functors $\varphi_* : A\mbox{-Mod} \to B\mbox{-Mod}$ (restriction) and
$\varphi^* :  B\mbox{-Mod} \to A\mbox{-Mod}$ (induction).
Furthermore, we let $\mu : A \otimes_B A \to A$ denote multiplication map 
defined by the relations $\mu(a \otimes a') = aa'$, for $a,a' \in A$.
Following \cite[p. 75]{brzezinski2005} we say that $A/B$ is separable
if $\mu$ has an $A$-bimodule section $\delta : A \to A \otimes_B A$.
In \cite[Proposition 1.3]{nastasescu1989} (see also \cite[Proposition 2.11]{kadison1999}) 
it was shown that if $A$ and $B$ are unital with $1_A = 1_B$, then $\varphi_*$ (or $\varphi^*$) is separable
if and only if $A/B$ is separable (or $\varphi$ splits as a $B$-bimodule map).
We now generalize these results to certain rings with enough idempotents. 

\begin{prop}\label{propsepidempotent}
If $A/B$ is an extension of rings with enough idempotents such that
$\{ e_i \}_{i \in I} \subseteq B$ is a complete set of idempotents for both $A$ and $B$, 
then the following are equivalent:
\begin{itemize}

\item[(i)] the functor $\varphi_*$ is separable;

\item[(ii)] the ring extension $A/B$ is separable:

\item[(iii)] for all $i \in I$ there exists an element $x_i \in \sum_{j \in I} e_i A e_j \otimes_B e_j A e_i$ such that 
for all $j \in I$ and all $a \in e_i A e_j$, the equalities $\mu(x_i)=e_i$ and $x_i a = a x_j$ hold.

\end{itemize}
\end{prop}

\begin{proof}
We adapt the first half of the proof of \cite[Proposition 1.3]{nastasescu1989} to our situation.

(i)$\Rightarrow$(ii): Suppose that $\varphi_*$ is separable. 
To $\mu : A \otimes_B A \to A$ we associate $\delta' : A \to A \otimes_B A$ defined by
$\delta'(a) = \sum_{i \in I} e_i \otimes a$ for $a \in A$.
Note that since $e_i \otimes a = e_i^2 \otimes a = e_i \otimes e_i a$ 
the sum defining $\delta'(a)$ is finite and hence well defined.
The map $\delta'$ is clearly right $A$-linear. Now we show that $\delta'$ is left $B$-linear.
To this end, take $a \in A$ and $b \in B$. Then 
$\delta'(ba) = 
\sum_{i \in I} e_i \otimes ba = 
\sum_{i \in I} e_i b \otimes a =
\left( \sum_{i \in I} e_i b \right) \otimes a =
b \otimes a =
\left( \sum_{i \in I} b e_i \right) \otimes a = 
b \sum_{i \in I} e_i \otimes a =
b \delta'(a)$.
If we put $\delta = \varphi_{A , A \otimes_B A} ( \delta' ) $, then, since $\mu \circ \delta' = {\rm id}_A$,
we get from separability of $\varphi_*$ that $\mu \circ \delta = {\rm id}_A$.
From the definition of the functor $\varphi_*$ it follows that $\delta$ is left $A$-linear.
We show that $\delta$ is also right $A$-linear.
In fact, suppose that $a \in A$ and let $\alpha_a : A \to A$ and
$\beta_a : A \otimes_B A \rightarrow A \otimes_B A$ denote 
the left $R$-linear maps which multiplies by $a$ from the right.
Since $\delta'$ is right $S$-linear it follows that $\delta' \circ \alpha_a = \beta_a \circ \delta'$.
Thus, from (SF2), it follows that $\delta \circ \alpha_a = \beta_a \circ \delta$,
since $\varphi_*( \alpha_a ) = \alpha_a$ and $\varphi_*( \beta_a ) = \beta_a$.
Therefore, $\delta$ is right $A$-linear.

(ii)$\Rightarrow$(iii):
Suppose that $\mu$ has an $A$-bimodule section $\delta : A \to A \otimes_B A$.
For all $i \in I$ put $x_i = \delta(e_i)$. Then $\mu(x_i) = (\mu \circ \delta)(e_i) = e_i$.
If $i,j \in I$ and $a \in e_i A e_j$, then
$x_i a = \delta(e_i) a 
      = \delta(e_i a) 
      = \delta(e_i a e_j) 
      = \delta(a e_j) 
      = a \delta(e_j) 
      = a x_j.$

(iii)$\Rightarrow$(i):
Suppose that for all $i \in I$ there exists an element
$x_i \in A \otimes_B A$ satisfying the conditions in (iii) above.
For all $i \in I$ we put
\[
x_i = \sum_{j,k \in I} y_{ij}^{(k)} \otimes z_{ji}^{(k)}
\]
for some $y_{ij}^{(k)} \in e_i A e_j$ and $z_{ji}^{(k)} \in e_j A e_i$ such that 
for all but finitely many $k \in I$, the relations $y_{ij}^{(k)} = z_{ji}^{(k)} = 0$ hold.
Suppose that $M$ and $N$ are left $A$-modules.
We define 
$
\varphi_{M,N}^{\varphi_*} : {\rm Hom}_B(\varphi_*(M),\varphi_*(N)) \to {\rm Hom}_A(M,N) 
$
in the following way. Take $f \in {\rm Hom}_B( \varphi_*(M) , \varphi_*(N) )$
and put $\varphi_{M,N}^{\varphi_*}(f) = \tilde{f}$ where 
\[
\tilde{f}(m) = \sum_{i,j,k \in I} y_{ij}^{(k)} f ( z_{ji}^{(k)} m )
\]
for $m \in M$. If $f \in {\rm Hom}_A( M , N )$ and $m \in M$, then 
\[
\tilde{f}(m) = \sum_{i,j,k \in I} y_{ij}^{(k)} f ( z_{ji}^{(k)} m ) =
\sum_{i,j,k \in I} y_{ij}^{(k)} z_{ji}^{(k)} f( m ) = 
\sum_{i \in I} e_i f(m) = f(m).
\]
Therefore (SF1) holds. Now we show (SF2). Suppose that  
\begin{equation}\label{comm}
\varphi_*(\beta) f = g \varphi_*(\alpha)
\end{equation}
where $M$, $N$, $M'$ and $N'$ are left $A$-modules,
$\alpha \in {\rm Hom}_A(M,M')$, $\beta \in {\rm Hom}_A(N,N')$,
$f \in {\rm Hom}_B( \varphi_*(M),\varphi_*(N) )$ and $g \in {\rm Hom}_B( \varphi_*(N'), \varphi_*(M') )$.
If $m \in M$, then, from (\ref{comm}), we get that
\[
\beta( \tilde{f} (m) ) = \beta \left( \sum_{i,j,k} y_{ij}^{(k)} f ( z_{ji}^{(k)} m ) \right) =
\sum_{i,j,k} y_{ij}^{(k)} \beta( f ( z_{ji}^{(k)} m ) ) 
\]
\[
= \sum_{i,j,k} y_{ij}^{(k)} g ( \alpha ( z_{ji}^{(k)} m ) )
= \sum_{i,j,k} y_{ij}^{(k)} g ( z_{ji}^{(k)} \alpha( m ) ) = \tilde{g} (\alpha(m)).
\]
Thus (SF2) holds.
\end{proof}

\begin{prop}\label{propsepinduction}
If $A/B$ is an extension of rings with enough idempotents such that
$\{ e_i \}_{i \in I} \subseteq B$ is a complete set of idempotents for both $A$ and $B$,
then the following are equivalent:
\begin{itemize}

\item[(i)] the functor $\varphi^*$ is separable;

\item[(ii)] the map $\varphi$ splits as a $B$-bimodule map;

\item[(iii)] for all $i,j \in I$ the induced map $\varphi_{ij} : e_i B e_j \to e_i A e_j$
splits as an $e_i B e_i$-$e_j B e_j$-bimodule map.

\end{itemize}
\end{prop}

\begin{proof}
We adapt the second half of the proof of \cite[Proposition 1.3]{nastasescu1989} to our situation.

(i)$\Rightarrow$(ii): Suppose that $\varphi^*$ is separable.
Define the map $\gamma : A \otimes_B A \to A \otimes_B B$ by
$\gamma(a \otimes a') = \sum_{i \in I} aa' \otimes e_i$ for $a,a' \in A$.
Note that since $aa' \otimes e_i = aa' \otimes e_i^2 = aa' e_i \otimes e_i$ the last sum is finite.
Then clearly $\gamma(1 \otimes \varphi) = {\rm id}_{A \otimes_B B}$. Thus from separability of $\varphi^*$
we get that $\varphi_{A,B}^{\varphi^*}(\gamma) \varphi = {\rm id}_B$.
By a variation of the proof of (i)$\Rightarrow$(ii) in Proposition \ref{propsepidempotent},
it follows that $\varphi_{A,B}^{\varphi^*}(\gamma)$ is a $B$-bimodule homomorphism.
Hence $\varphi$ splits as a $B$-bimodule homomorphism.

(ii)$\Rightarrow$(iii): Suppose that $\varphi$ is split by an $R$-bilinear map $\psi$.
For all $i,j \in I$ define $\psi_{ij} : e_i A e_j \to e_i B e_j$ by $\psi_{ij}(a) = e_i \psi(a) e_j$ for $a \in e_i A e_j$.
Take $a \in A$, $b,b' \in B$ and $i,j \in I$. Then 
$\psi_{ij}( e_i b e_i a e_j b' e_j ) = e_i \psi( e_i b e_i a e_j b' e_j) e_j = e_i b e_i \psi(a) e_j b' e_j$.
Therefore $\psi_{ij}$ is an $e_i B e_i$-$e_j B e_j$-bimodule map.
Also 
\[
\psi_{ij} ( \varphi_{ij} ( e_i a e_j ) ) = e_i \psi ( e_i \varphi(a) e_j ) e_j = e_i \psi(\varphi(a)) e_j = e_i a e_j
\]
which implies that $\varphi_{ij}$ splits as an $e_i B e_i$-$e_j B e_j$-bimodule map.

(iii)$\Rightarrow$(i): Suppose that for all $i,j \in I$ the induced $e_i B e_i$-$e_j B e_j$-bimodule map 
$\varphi_{ij} : e_i B e_j \to e_i A e_j$ is split by a map $\psi_{ij}$.
Suppose that $M$ and $N$ are left $B$-modules.
We define 
$
\varphi_{M,N}^{\varphi^*} : {\rm Hom}_A(\varphi^*(M),\varphi^*(N)) \to {\rm Hom}_B(M,N) 
$
in the following way. Take $f \in {\rm Hom}_A( \varphi^*(M) , \varphi^*(N) )$
and put $\varphi_{M,N}^{\varphi^*}(f) = \tilde{f}$ where 
\[
\tilde{f}(m) = \sum_{i,j \in I} ( \rho \circ ( \psi_{ij} \otimes {\rm id}_N ) )( f(e_i \otimes m) ),
\]
for $m \in M$, where $\rho : B \otimes_B N \rightarrow N$ denotes the multiplication map.
If $f = \varphi^*(g) = {\rm id}_A \otimes g$ for some $g \in {\rm Hom}_B( M , N )$ and $m \in M$, then 
\[
\tilde{f}(m) = \sum_{i,j \in I} ( \rho \circ ( \psi_{ij} \otimes {\rm id}_N ) )( e_i \otimes g(m) ) =
\sum_{i,j \in I} e_i e_j g(m) = \sum_{i \in J} g(e_i m) = g(m). 
\]
Therefore (SF1) holds. Now we show (SF2). Suppose that  
\begin{equation}\label{commagain}
\varphi^*(\beta) f = g \varphi^*(\alpha)
\end{equation}
where $M$, $N$, $M'$ and $N'$ are left $B$-modules,
$\alpha \in {\rm Hom}_B(M,M')$, $\beta \in {\rm Hom}_B(N,N')$,
$f \in {\rm Hom}_A( \varphi^*(M),\varphi^*(N) )$ and $g \in {\rm Hom}_A( \varphi^*(N'), \varphi^*(M') )$.
If $m \in M$, then, from (\ref{commagain}), we get that
\[
\beta( \tilde{f} (m) ) = \beta \left(  \sum_{i,j \in I} ( \rho \circ ( \psi_{ij} \otimes {\rm id}_N ) )( f(e_i \otimes m) )  \right) 
\]
\[
= \sum_{i,j \in I} \rho \circ ( \psi_{ij} \otimes {\rm id}_N ) \circ \varphi^*(\beta) \circ f (e_i \otimes m) 
\]
\[
= \sum_{i,j \in I} \rho \circ ( \psi_{ij} \otimes {\rm id}_N ) \circ g \circ \varphi^*(\alpha) (e_i \otimes m) 
\]
\[
= \sum_{i,j \in I} ( \rho \circ ( \psi_{ij} \otimes {\rm id}_N ) )( g(e_i \otimes \alpha(m)) )  = \tilde{g} (\alpha(m)).
\]
Thus (SF2) holds.
\end{proof}

\subsection*{Strongly graded rings}

From now on, let $\G$ be a groupoid and suppose that $R$ is an object unital strongly $\G$-graded ring.

\begin{prop}\label{propR0inR}
The induction functor $\varphi^* : R_0\mbox{-Mod} \to R\mbox{-Mod}$,
associated to the inclusion map $\varphi : R_0 \to R$, is separable. 
\end{prop}

\begin{proof}
Take $e,f \in \G$. Define $\psi_{e,f} : R \to R_0$ in the following way.
Take $\sigma \in \G$ and $r_{\sigma} \in R_{\sigma}$. If $e \neq f$, then put $\psi_{e,f}(r_{\sigma}) = 0$.
Put $\psi_{e,e}(r_{\sigma})=0$, if $r(\sigma) \neq e$ or $r(\sigma) \neq e$.
Finally, if $d(\sigma)=e=r(\sigma)$, then put $\psi_{e,e}(r_{\sigma}) = r_{\sigma}$.
It is clear that $\psi_{e,f}$ is an $R_e$-$R_f$-bimodule splitting of $\varphi_{e,f}$.
The claim now follows from Proposition \ref{propsepinduction}. 
\end{proof}

\begin{rem}
It is easy to see that the conclusion in Proposition \ref{propR0inR} holds for any 
object unital (not necessarily strongly) $\G$-graded ring. 
\end{rem}

We now turn to the question of separability of the restriction functor 
$\varphi_* : R\mbox{-Mod} \to R_0\mbox{-Mod}$
associated to the inclusion map $\varphi : R_0 \to R$.
To this end, we need some more notions and results.

\begin{defi}
Take $\si \in \G$. 
From the relation $1_{R_r(\si)} \in R_\si R_{\si^{-1}}$ it follows that there is $n_\si \in \N$,
and $u_\si^{(i)} \in R_\si$ and $v_{\si^{-1}}^{(i)} \in R_{\si^{-1}}$, for $i \in \{ 1,\ldots,n_\si \}$,
such that
\[
1_{R_{r(\si)}} = \sum_{i=1}^{n_\si} u_\si^{(i)} v_{\si^{-1}}^{(i)}.
\]
Unless otherwise stated, the elements $u_\si^{(i)}$ and $v_{\si^{-1}}^{(i)}$ are fixed.
We also assume that if $e \in \G_0$, then $n_e = 1$ and that $u_e^{(1)} = v_e^{(1)} = 1_{R_e}$.
Define the additive map $\gamma_\si : R \to R$ by 
\[
\gamma_\si(r) =  \sum_{i=1}^{n_\si} u_\si^{(i)} r v_{\si^{-1}}^{(i)}, 
\]
for $r \in R.$
\end{defi}

\begin{prop}\label{firstrestriction}
Take $\si \in \G$. 
The additive map $\gamma_\si : R \to R$
restricts to an additive map  $R_0 \to R_{r(\si)}$.
This map, in turn, restricts to an additive map $R_{d(\si)} \to R_{r(\si)}$.
\end{prop}

\begin{proof}
Take $r \in R_0$. Then, since $u_{\si}^{(i)} r_e = 0$, for $e \in \G_0 \setminus \{ d(\si) \}$, we get that
\begin{eqnarray*}
\gamma_\si(r) &=& \sum_{i=1}^{n_\si} u_\si^{(i)} r v_{\si^{-1}}^{(i)} 
               = \sum_{i=1}^{n_\si} \sum_{e \in \G_0} u_\si^{(i)} r_e v_{\si^{-1}}^{(i)} \\
              &=& \sum_{i=1}^{n_\si} u_\si^{(i)} r_{d(\si)} v_{\si^{-1}}^{(i)} 
               = \gamma_\si(r_{d(\si)}).
\end{eqnarray*}
Since $u_\si^{(i)} r_{d(\si)} v_{\si^{-1}}^{(i)} \in R_\si R_{d(\si)} R_{\si^{-1}} \subseteq R_{r(\si)}$ the claim follows.
\end{proof}

\begin{rem}\label{remarkcenter}
It is clear that $Z(R_0) = \oplus_{e \in \G_0} Z(R_e)$.
\end{rem}

\begin{prop}
If $\si \in \G$ and $r \in Z(R_0)$, then the definition of $\gamma_{\si}(r)$ does not depend 
on the choice of the elements $u_{\si}^{(i)}$ and $v_{\si^{-1}}^{(i)}$.
\end{prop}

\begin{proof}
Take $m_{\si} \in \mathbb{N}$,
$s_{\si}^{(j)} \in R_{\si}$ and $t_{\si^{-1}}^{(j)} \in R_{\si^{-1}}$,
for $j \in \{1,\ldots,m_{\si} \}$, such that 
$\sum_{j=1}^{m_{\si}} s_{\si}^{(j)} t_{\si^{-1}}^{(j)} = 1_{R_{r(\si)}}$. Then
\[
\gamma_{\si}(r) =
\sum_{i=1}^{n_{\si}} u_{\si}^{(i)} r v_{\si^{-1}}^{(i)} 
=  \sum_{i=1}^{n_{\si}} 1_{R_{r(\si)}} u_{\si}^{(i)} r v_{\si^{-1}}^{(i)} 
= \sum_{i=1}^{n_{\si}} \sum_{j=1}^{m_{\si}} s_{\si}^{(j)} t_{\si^{-1}}^{(j)} u_{\si}^{(i)} r v_{\si^{-1}}^{(i)}.
\]
Since $t_{\si^{-1}}^{(j)} u_{\si}^{(i)} \in R_0$ and $r \in Z(R_0)$, the last sum equals
\[
\sum_{i=1}^{n_{\si}} \sum_{j=1}^{m_{\si}} s_{\si}^{(j)} r t_{\si^{-1}}^{(j)} u_{\si}^{(i)}  v_{\si^{-1}}^{(i)}
 =   \sum_{j=1}^{m_{\si}} s_{\si}^{(j)} r t_{\si^{-1}}^{(j)} 1_{R_{r(\si)}} 
 =   \sum_{j=1}^{m_{\si}} s_{\si}^{(j)} r t_{\si^{-1}}^{(j)}.
\]
\end{proof}

\begin{prop}\label{composition}
If $\si,\tau \in \G$ and $r \in Z(R_0)$, then
\begin{itemize}

\item[(a)] if $(\si,\tau) \notin \G_2$, then $\gamma_\si ( \gamma_\tau (r) ) = 0$, and

\item[(b)] if $(\si,\tau) \in \G_2$, then $\gamma_\si( \gamma_\tau ( r ) ) = \gamma_{\si\tau}(r_{d(\tau)}) 1_{R_{r(\si)}}$.

\end{itemize}
\end{prop}

\begin{proof}  
Take $\si,\tau \in \G$ and $r \in Z(R_0)$.
From the definitions of $\gamma_\si$ and $\gamma_\tau$, it follows that
\[
\gamma_\si ( \gamma_\tau ( r ) ) 
= \sum_{i=1}^{n_\tau} \gamma_\tau ( u_\tau^{(i)} r v_{\tau^{-1}}^{(i)} ) 
= \sum_{i=1}^{n_\tau} \sum_{j=1}^{n_\si} u_\si^{(j)} u_\tau^{(i)} r v_{\tau^{-1}}^{(i)} v_{\si^{-1}}^{(j)}.
\]
If $(\si,\tau) \notin \G_2$, then $u_\si^{(j)} u_\tau^{(i)}=$ so that $\gamma_\si ( \gamma_\tau (r) ) = 0$.
Now suppose that $(\si,\tau) \in \G_2$.
Since $u_\si^{(j)} u_\tau^{(i)} \in R_{\si\tau}$, the last sum equals
\begin{displaymath}
\sum_{i=1}^{n_\tau} \sum_{j=1}^{n_\si} 
1_{R_{r(\si\tau)}} u_\si^{(j)} u_\tau^{(i)} r v_{\tau^{-1}}^{(i)} v_{\si^{-1}}^{(j)} =
\sum_{i=1}^{n_\si} \sum_{j=1}^{n_\si} \sum_{k=1}^{n_{\si\tau}}  
u_{\si\tau}^{(k)} v_{\tau^{-1} \si^{-1}}^{(k)}
u_\si^{(j)} u_\tau^{(i)} r v_{\tau^{-1}}^{(i)} v_{\si^{-1}}^{(j)}.
\end{displaymath}
Since $r \in Z(R_0)$ and 
$v_{\tau^{-1} \si^{-1}}^{(k)} u_\si^{(j)} u_\tau^{(i)} \in R_0$, the last sum equals
\begin{displaymath}
\sum_{i=1}^{n_\tau} \sum_{j=1}^{n_\si} \sum_{k=1}^{n_{\si\tau}}  
u_{\si\tau}^{(k)} r v_{\tau^{-1} \si^{-1}}^{(k)}
u_\si^{(j)} u_\tau^{(i)} v_{\tau^{-1}}^{(i)} v_{\si^{-1}}^{(j)} = 
\sum_{j=1}^{n_\si} \sum_{k=1}^{n_{\si\tau}}  
u_{\si\tau}^{(k)} r v_{\tau^{-1} \si^{-1}}^{(k)}
u_\si^{(j)} 1_{R_{r(\tau)}} v_{\tau^{-1}}^{(j)}.
\end{displaymath}
Since $d(\si)=r(\tau)$, we get that
the last sum equals
\begin{displaymath}
\sum_{j=1}^{n_\si} \sum_{k=1}^{n_{\si\tau}}  
u_{\si\tau}^{(k)} r v_{\tau^{-1} \si^{-1}}^{(k)}
u_\si^{(j)} v_{\si^{-1}}^{(j)} = 
\sum_{k=1}^{n_{\si\tau}}  
u_{\si\tau}^{(k)} r v_{\tau^{-1} \si^{-1}}^{(k)} 1_{R_{r(\si)}} = 
\gamma_{\si\tau}(r) 1_{R_{r(\si)}}.
\end{displaymath}
\end{proof}

\begin{prop}\label{restriction}
Take $\si \in \G$. 
The additive map $\gamma_\si : R \to R$
restricts to a surjective ring homomorphism 
$Z(R_0) \rightarrow Z(R_{r(\si)})$.
This map, in turn, restricts to a ring isomorphism
$Z(R_{d(\si)}) \rightarrow Z(R_{r(\si)})$ such that
if $(\si,\tau) \in \G_2$ and $r \in R_{d(\tau)}$, then $\gamma_\si ( \gamma_\tau ( r ) ) = \gamma_{\si \tau}(r)$.
\end{prop}

\begin{proof}
From Proposition \ref{firstrestriction} it follows that $\gamma_\si(Z(R_0)) \subseteq R_{r(\si)}$.
Using Remark \ref{remarkcenter}, we only need to show that $\gamma_\si(Z(R_0)) \subseteq Z(R_0)$.
To this end, take $r \in Z(R_0)$ and $r' \in R_0$. Then,
since $v_{\si^{-1}}^{(i)} r' \in R_{\si^{-1}}$, we get that
\[
\gamma_\si(r) r' = \sum_{i=1}^{n_\si} u_\si^{(i)} r v_{\si^{-1}}^{(i)} r' 
= \sum_{i=1}^{n_\si} u_\si^{(i)} r v_{\si^{-1}}^{(i)} r' 1_{R_{r(\si)}} 
= \sum_{i=1}^{n_\si} \sum_{j=1}^{n_\si} u_\si^{(i)} r v_{\si^{-1}}^{(i)} r' u_\si^{(j)} v_{\si^{-1}}^{(j)}.
\]
Since $r \in Z(R_0)$, $v_{\si^{-1}}^{(i)} r' u_\si^{(j)} \in R_0$
and $r' u_\si^{(j)} \in R_\si$, the last sum equals
\[
\sum_{i=1}^{n_\si} \sum_{j=1}^{n_\si} u_\si^{(i)}  v_{\si^{-1}}^{(i)} r' u_\si^{(j)} r v_{\si^{-1}}^{(j)} 
= \sum_{j=1}^{n_\si} 1_{R_{r(\si)}} r' u_\si^{(j)} r v_{\si^{-1}}^{(j)}
= \sum_{j=1}^{n_\si} r' u_\si^{(j)} r v_{\si^{-1}}^{(j)} 
= r' \gamma_\si(r).
\]
This shows that $\gamma_\si(r) \in Z(R_0)$.
Now we show that the restriction of $\gamma_\si$ to $Z(R_0)$
respects multiplication.
Take $r,r' \in Z(R_0)$. Then
\begin{eqnarray*}
\gamma_\si(r r') &=& \sum_{i=1}^{n_\si} u_\si^{(i)} r r' v_{\si^{-1}}^{(i)} 
                 = \sum_{i=1}^{n_\si} 1_{R_{r(\si)}} u_\si^{(i)} r r' v_{\si^{-1}}^{(i)} \\
                 &=& \sum_{i=1}^{n_\si} \sum_{j=1}^{n_\si} u_\si^{(j)} v_{\si^{-1}}^{(j)} u_\si^{(i)} r r' v_{\si^{-1}}^{(i)}.
\end{eqnarray*}
Since $r \in Z(R_0)$ and $v_{\si^{-1}}^{(j)} u_\si^{(i)} \in R_0$, the last sum equals
\[
\sum_{i=1}^{n_\si} \sum_{j=1}^{n_\si} u_\si^{(j)} r v_{\si^{-1}}^{(j)} u_\si^{(i)} r' v_{\si^{-1}}^{(i)} 
=\sum_{j=1}^{n_\si} u_\si^{(j)} r v_{\si^{-1}}^{(j)} \sum_{i=1}^{n_\si} u_\si^{(i)} r' v_{\si^{-1}}^{(i)} 
= \gamma_\si(r) \gamma_\si(r').
\]
Next, we show that the restriction $Z(R_0) \rightarrow Z(R_{r(\si)})$ is surjective. Take $r \in Z(R_{r(\si)})$.
From Proposition \ref{composition}, we get that
\[
\gamma_\si ( \gamma_{\si^{-1}} ( r_{r(\si)} ) ) = \gamma_{r(\si)}(r_{r(\si)}) 1_{R_{r(\si)}} 
= r_{r(\si)} 1_{R_{r(\si)}} = r_{r(\si)}.
\]
We also need show that the restriction $Z(R_{d(\si)}) \rightarrow Z(R_{r(\si)})$
is injective. Suppose that $r \in Z(R_{d(\si)})$ is chosen so that 
$\gamma_\si( r ) = 0$.
From Proposition \ref{composition}, we get that
$r = \gamma_{d(\si)}( r ) = \gamma_{\si^{-1}} ( \gamma_\si ( r ) ) = 0$.
Finally, we show that each $\gamma_\si$ respects multiplicative identities
\[
\gamma_\si( 1_{R_{d(\si)}} ) = 
\sum_{i=1}^{n_\si} u_\si^{(i)} 1_{R_{d(\si)}} v_{\si^{-1}}^{(i)} = 
\sum_{i=1}^{n_\si} u_\si^{(i)} v_{\si^{-1}}^{(i)} = 1_{R_{c(\si)}}.
\]
\end{proof}

\begin{rem} Let $D_\sigma=Z(R_{r(\sigma)}).$ Then it follows from 
Proposition \ref{restriction} that the family $\{\gamma_\sigma: D_{\sigma\m}\to  D_\sigma \}_{\sigma\in G}$ 
is a global groupoid action of $\G$ on $Z(R_0),$ (see \cite[p. 3660]{BP}).
\end{rem}

\begin{prop}
If $\si \in \G$, $r \in R_\si$ and $x \in Z(R_{d(\si)})$, then $r x = \gamma_\si(x) r$.
\end{prop}

\begin{proof}
Since $x \in Z(R_{d(\si)})$ and $v_{\si^{-1}}^{(i)} r \in R_{d(\si)}$, we get that
\[
r x = 1_{R_{r(\si)}} r x 
    = \sum_{i=1}^{n_\si} u_\si^{(i)} v_{\si^{-1}}^{(i)} r x 
    = \sum_{i=1}^{n_\si} u_\si^{(i)} x v_{\si^{-1}}^{(i)} r  
    = \gamma_\si(x) r.
\]
\end{proof}

\begin{prop}\label{uniqueness}
Take $\si \in \G$, $s \in Z( R_{d(\si)} )$ and $t \in Z( R_{c(\si)} )$.
If for all $r \in R_{\si}$, the equality $t r = r s$ holds, then $t = \gamma_\si(s)$.
\end{prop}

\begin{proof}
This follows from
\[
t = t 1_{R_{r(\si)}} 
  = \sum_{i=1}^{n_\si} t u_\si^{(i)} v_{\si^{-1}}^{(i)} 
  = \sum_{i=1}^{n_\si} u_\si^{(i)} s v_{\si^{-1}}^{(i)} 
  = \gamma_\si(s).
\]
\end{proof}

\begin{defi}\label{definition:trace}
If $e \in \G_0$ and $\G(e)$ is finite, 
then we define the {\it trace function at $e$},
$
{\rm tr}_{e} : Z(R_0) \rightarrow Z(R_0) 
$,
by 
$
{\rm tr}_{e}(r) = \sum_{\si \in \G(e)} \gamma_\si(r)
$
for $r \in Z(R_0)$. 
Note that if $\G$ is a finite group with identity element $e$, then ${\rm tr_e}$ coincides with
the usual trace $Z(R_e) \to Z(R_e)$ from the group graded case (see \cite{nastasescu1989}).
\end{defi}

\begin{defi}
For all $\si \in \G$ we put $w_\si = \sum_{i=1}^{n_\si} u_\si^{(i)} \otimes v_{\si^{-1}}^{(i)} \in R_\si \otimes_{R_0} R_{\si^{-1}}$.
\end{defi}

\begin{prop}\label{niceequality}
If $(\si,\ta) \in \G_2$ and $r \in R_\si$, then $r w_\ta = w_{\si\ta} r$.
\end{prop}

\begin{proof}
Since $r u_\ta^{(i)} \in R_{\si\ta}$, we get that
\begin{eqnarray*}
r w_\ta &=& \sum_{i=1}^{n_\ta} r u_\ta^{(i)} \otimes v_{\ta^{-1}}^{(i)} 
        = \sum_{i=1}^{n_\ta} 1_{R_{r(\si\ta)}} r u_\ta^{(i)} \otimes v_{\ta^{-1}}^{(i)} \\
        &=& \sum_{i=1}^{n_\ta} \sum_{j=1}^{n_{\si\ta}} u_{\si\ta}^{(j)} v_{(\si\ta)^{-1}}^{(j)} r u_\ta^{(i)} \otimes v_{\ta^{-1}}^{(i)} 
        = \sum_{i=1}^{n_\ta} \sum_{j=1}^{n_{\si\ta}} u_{\si\ta}^{(j)} \otimes v_{(\si\ta)^{-1}}^{(j)} r u_\ta^{(i)} v_{\ta^{-1}}^{(i)} \\
        &=& \sum_{j=1}^{n_{\si\ta}} u_{\si\ta}^{(j)} \otimes v_{(\si\ta)^{-1}}^{(j)} r 1_{R_{r(\ta)}} 
        = w_{\si\ta} r.
\end{eqnarray*}
\end{proof}

\begin{defi}
Define an equivalence relation $\sim$ on $\G_0$ by saying that $e \sim f$, for $e,f \in \G_0$,
if there exists $\sigma \in \G$ with $d(\si) = e$ and $r(\si) = f$.
For all $e \in \G_0$, let $[e] = \{ f \in \G_0 \mid f \sim e \}$ denote the equivalence class to which $e$ belongs.
Let $\Omega$ denote a set of representatives for the different equivalence classes in $\G_0$ defined by $\sim$. 
If $e,f \in \G_0$, then we put $\G(e,f) = \{ \si \in \G \mid d(\si)=e \ {\rm and} \ r(\si)=f \}$.
Note that $\G(e,e)$ coincides with $\G(e)$ defined earlier.
\end{defi}

Now we prove a result the conclusions of which implies Theorem \ref{maintheorem1}.

\begin{prop}\label{separableprop}
The following are equivalent:
\begin{itemize}

\item[(i)] the restriction functor 
$\varphi_* : R\mbox{-Mod} \to R_0\mbox{-Mod}$
associated to the inclusion map $\varphi : R_0 \to R$, is separable;

\item[(ii)] the ring extension $R/R_0$ is separable;

\item[(iii)] for all finite subsets $E$ of $\G_0$, the extension $R_{\G(E)} / R_E$ is separable;

\item[(iv)] for all $e \in \G_0$, the ring extension $R_{\G(e)} / R_e$ is separable;

\item[(v)] for all $e \in \G_0$, the group $\G(e)$ is finite and $1_{R_e} \in {\rm tr}_e(Z(R_e))$.

\end{itemize}
\end{prop}

\begin{proof}
(i)$\Leftrightarrow$(ii): This follows from Proposition \ref{propsepidempotent}.

(ii)$\Rightarrow$(v): Take $e \in \G_0$. 
From Proposition \ref{propsepidempotent}, it follows that there exists
$y_e \in R \otimes_{R_0} R$ such that $\mu(y_e) = 1_{R_e}$ and 
for all $f \in \G_0$ and all $a \in 1_{R_e} R 1_{R_f}$, the equality $y_e a = a y_f$ holds. 
Then there exists $m_e,k_\si \in \mathbb{N}$ such that 
\[
y_e = \sum_{i=1}^{m_e} \sum_{\si \in \G(f_i,e)} z_\si^{(i)}
\]
where
\[
z_\si^{(i)} := \sum_{j=1}^{k_\si} a_\si^{(j)} \otimes b_{\si^{-1}}^{(j)} = 0
\]
for all but finitely many $\si \in \G(f_i,e)$.
For all $i \in \{ 1,\ldots,m_e \}$ and all $\si \in \G(f_i,e)$ put
$c_\si^{(i)} = \mu( z_\si^{(i)} ).$
Take $r \in R_e$. Then, from the equality $y_e r = r y_e$ it follows that
$z_\si^{(i)} r = r z_\si^{(i)}$, for all $i$ and all $\si$. Thus
\[
c_\si^{(i)} r = \mu( z_\si^{(i)} ) r = \mu( z_\si^{(i)} r ) =
 \mu( r z_\si^{(i)} ) = r \mu( z_\si^{(i)} ) = r c_\si^{(i)}.
\]
Hence $c_\si^{(i)} \in Z(R_e)$.
Take $\tau \in \G(e)$ and $r_\tau \in R_{\tau}$. The equality
$y_e r_\tau = r_\tau y_e$ implies that $z_{\tau \si}^{(i)} r_\tau = r_\tau z_\si^{(i)}$
from which it follows that $c_{\tau \si}^{(i)} r_\tau = r_\tau c_\si^{(i)}$.
From Proposition \ref{uniqueness} we get that $\gamma_\tau(c_\si^{(i)}) = c_{\tau \si}^{(i)}$.
Since $1_{R_e} \neq 0$, at least one the $c_\si^{(i)}$ must be non-zero.
Thus $\G(e)$ is finite. Now, for all $i \in \{ 1,\ldots,m_e \}$ take $\si_i \in \G(f_i,e)$. 
Put $d_e = \sum_{i=1}^{m_e} c_{\si_i}^{(i)}$. Then, $d_e \in Z(R_0)$ and
\begin{eqnarray*}
{\rm tr}_e \left( d_e \right) &=& \sum_{i=1}^{m_e} {\rm tr}( c_{\si_i}^{(i)} ) 
 = \sum_{i=1}^{m_e} \sum_{\tau \in \G(e,e)} \gamma_\tau( c_{\si_i}^{(i)} ) 
 = \sum_{i=1}^{m_e} \sum_{\tau \in \G(e,e)} c_{\tau \si_i}^{(i)} \\
 &=& \sum_{i=1}^{m_e} \sum_{\si \in \G(f_i,e)} c_\si^{(i)} 
 = \sum_{i=1}^{m_e} \sum_{\si \in \G(f_i,e)} \mu( z_\si^{(i)} ) 
 = \mu( y_e ) = 1_{R_e}.
\end{eqnarray*}

(v)$\Rightarrow$(ii): For all $\omega \in \Omega$ take $r_\omega \in R_\omega$ such that ${\rm tr}_\omega(r_\omega) = 1_{R_\omega}$.
For all $e \in [\omega]$ put $x_e = \sum_{\si \in \G(\omega,e)} \gamma_\si (r_\omega) w_\si$.
First we show condition (i) from Proposition \ref{propsepidempotent}.
Fix $\tau \in \G(\omega,e)$. Then
\begin{eqnarray*}
\mu(x_e) &=& \sum_{\si \in \G(\omega,e)} \gamma_\si (r_\omega) \mu( w_\si ) 
         = \sum_{\si \in \G(\omega,e)} \gamma_\si (r_\omega) 1_{R_e} \\
         &=& \sum_{\si \in \G(\omega,e)} \gamma_\si (r_\omega) 
         = \sum_{\si \in \G(\omega,e)} \gamma_{\tau} ( \gamma_{\tau^{-1} \si} (r_\omega) ) \\
         &=& \gamma_{\tau} \left(   \sum_{\si \in \G(\omega,e)} \gamma_{\tau^{-1} \si} (r_\omega)  \right) 
         = \gamma_{\tau} \left(   \sum_{\rho \in \G(\omega,\omega)} \gamma_{\rho} (r_\omega)  \right) \\
         &=& \gamma_{\tau} \left( {\rm tr}_\omega (r_\omega)  \right) 
         = \gamma_{\tau} ( 1_{R_{\omega}} ) 
         = 1_{R_e}.
\end{eqnarray*}
Next we show condition (ii) from Proposition \ref{propsepidempotent}. To this end, take $e,f \in \G_0$
and $s \in 1_{R_e} R 1_{R_f}$. We wish to show the equality $x_e s = s x_f$.
Case 1: $e \nsim f$. Then $1_{R_e} R 1_{R_f} = \{ 0 \}$ so the equality trivially holds.
Case 2: $e \sim f$. Take $\omega \in \Omega$ with $e,f \in [\omega]$.
It is enough to show the equality $x_e s = s x_f$ for $s \in R_\rho$ for all $\rho \in G(e,f)$.
For such an element $s$ we get, from Proposition \ref{niceequality}, that
\begin{eqnarray*}
s x_e &=& \sum_{\si \in \G(\omega,e)} s \gamma_\si( r_\omega ) w_\si 
       = \sum_{\si \in \G(\omega,e)} \gamma_{\si\rho}( r_\omega ) s w_\si \\
      &=& \sum_{\si \in \G(\omega,e)} \gamma_{\si\rho}( r_\omega ) w_{\rho\si} s 
       = \sum_{\tau \in \G(\omega,f)} \gamma_\tau( r_\omega ) w_{\tau} s 
       = x_f s.
\end{eqnarray*}  
From Proposition \ref{propsepidempotent} it follows that $R/R_0$ is separable. 

(i)$\Leftrightarrow$(iii) and (i)$\Leftrightarrow$(iv): 
Since $\G(E)$ is a groupoid of and $\G(e)$ is a group both of these equivalences follow from the equivalences 
(i)$\Leftrightarrow$(ii)$\Leftrightarrow$(v).
\end{proof}

\begin{rem}
Proposition \ref{separableprop} generalizes \cite[Proposition 4]{lundstrom2005} to the 
object unital graded situation. Moreover, our result greatly simplifies the condition of a ``global''
trace (see \cite[Definition 5]{lundstrom2005}) to a collection of ``local'' traces.
\end{rem}

\subsection*{Crossed products}

Suppose from now on that $R = A \rtimes^{\alpha}_{\beta} \G$ is an object crossed product
defined by a crossed system $( A , \G , \alpha , \beta )$. 
From Definition \ref{definition:trace}, it follows that if $e \in \G_0$ and $\G(e)$ is finite, 
then the trace function at $e$ is given by ${\rm tr}_{e}(a) = \sum_{\si \in \G(e)} \alpha_\si(a)$ for $a \in Z(A_e)$. 

\begin{prop}[Theorem \ref{maintheorem2}]\label{separablecrossed}
The ring extension $R/R_0$ is separable, if and only if, for all $e \in \G_0$, the group $\G(e)$ is finite and 
there exists $a \in A_e$ such that $\sum_{\si \in \G(e)} \alpha_\si(a) = 1_{A_e}$.
\end{prop}

\begin{proof}
This follows from Proposition \ref{separableprop}.
\end{proof}

For a finite set $X$, we let $|X|$ denote the cardinality of $X$.

\begin{prop}\label{separabletwisted}
If $R = B \rtimes_{\beta} \G$ is an object twisted groupoid ring, then $R/R_0$ is separable,
if and only if, for all $e \in \G_0$, the group $\G(e)$ is finite and $|\G(e)| \in U(B)$.
\end{prop}

\begin{proof}
This follows from Proposition \ref{separablecrossed}.
\end{proof}

Let $I$ denote a set. Define a groupoid structure on
$\G = I \times I$ in the following way.
If $(i,j) \in \G$, then put $(i,j)^{-1} = (j,i)$.
If $(i,j),(k,l) \in \G$, then $(i,j)(k,l)$ is defined
precisely when $j=k$ and is in that case equal to $(i,l)$. 
Let $B$ be a unital ring. The corresponding groupoid ring $B[\G]$
of $\G$ over $B$ is the ring of finite matrices over the index set $I$.
Note that $B[\G]_0 = B \G_0$

\begin{prop}\label{separablematrices}
If $\G = I \times I$, then $B[\G] / B \G_0$ is separable.
\end{prop}

\begin{proof}
This follows from Propostion \ref{separabletwisted}, since for all $i \in I$
we have that $|\G(i,i)| = 1$ which, of course, is invertible in $B$.
\end{proof}

\begin{prop}[Theorem \ref{maintheorem3}]\label{maincrossed}
Suppose that $L/K$ is a separable (not necessarily finite) field extension. Then the object crossed
product $R=(L/K,\beta)$ is separable over $R_0$, if and only if, ${\rm Aut}_K(L)$ is finite. 
In particular, if $L/K$ is Galois, then $R$ is separable over $R_0 = L$, if and only if, $L/K$ is finite.
\end{prop}

\begin{proof}
The ``only if'' statement follows from Proposition \ref{separableprop}. 
Now we show the ``if'' statement.
Put $H = {\rm Aut}_K(L)$ and suppose that $H$ is finite.
Since $L/K$ is separable it follows that $L/L^H$ also is separable.
Thus the trace map $L \to L^H$ is surjective (see e.g. \cite[Theorem (4.6)]{reiner1975}) and hence,
from Proposition \ref{separableprop} again, it follows that $R/R_0$ is separable.
The ``Galois'' part of the proof is immediate.
\end{proof}

\begin{ex}
It is easy to construct non-Galois separable field extensions $L/K$ of infinite degree
such that ${\rm Aut}_K(L)$ is a finite set. 
In fact, such an extension can be chosen so that ${\rm Aut}_K(L)$ is the trivial group (and hence finite).
For instance,  put $K = \mathbb{Q}$ and $L = \mathbb{Q}( \sqrt[3]{p_1},\sqrt[3]{p_2},\sqrt[3]{p_3} , \ldots )$,
where $p_1=2, \ p_2=3, \ p_3=5, \ldots$ denote the prime numbers. 
Then $L/K$ is of infinite degree and separable (since ${\rm char}(K)=0$).
It is easy to check that ${\rm Aut}_K(L) = \{ {\rm id}_L \}$.
Since $L$ is a real field none of the polynomials $\{ X^3 - p_i \}_{i=1}^{\infty}$ split in $L$.
Thus, $L/K$ is not normal and thus non-Galois.
\end{ex}


\begin{thebibliography}{R}

\bibitem{AnM}P. N. \'Anh and L. M\'arki, \emph{Morita equivalence for rings without identity,}  Tsukuba J.
Math. 11(2) (1987), 1–16.

\bibitem{BP} D. Bagio and A. Paques, Partial Groupoid Actions: Globalization, Morita Theory, and Galois Theory, \emph{Comm. Algebra}, \textbf{40} (2012), 3658-3678.

\bibitem{B-GT} G. B\"{o}hm, J.G\'omez-Torrecillas, \emph{Firm monads and firm Frobenius algebras}, Bull. Math. Soc. Sci. Math.  Roumanie 56 (104), 281-298 (2013).

\bibitem{brzezinski2005}
T. Brzezinski, L. Kadison and R. Wisbauer,
On coseparable and biseparable corings, in
{\it Hopf Algebras in Noncommutative Geometry and Physics}, Pure and applied mathematics,
volume 239, Marcel Dekker, New York (2005)

\bibitem{ELGO} L. El Kaoutit and J.G\'omez-Torrecillas, \emph{Invertible unital bimodules over rings with local units, and
related exact sequences of groups, II} J.  Algebra
 370 (2012), Pages 266-296.
\bibitem{Fu} K. R. Fuller, \emph{On rings whose left modules are direct sums of finitely generated modules}, Proc. Amer. Math. Soc. 54 (1976), 39–44.

\bibitem{kadison1999}
L. Kadison, New Examples of Frobenius Extensions, University Lecture Series, Vol. 14, Amer. Math. Soc.,
Providence, RI (1999)



\bibitem{L} M. V. Lawson, \emph{Inverse Semigroups. The Theory of Partial Symmetries} (World Scientific Pub. Co., 1995).

\bibitem{liu2005}
G. Liu, F. Li,
On Strongly Groupoid Graded Rings and the
Corresponding Clifford Theorem,
{\it   Alg. Colloquium} \textbf{13}(2), 181--196 (2005).

\bibitem{lundstrom2004}
P. Lundstr\"{o}m,
The category of groupoid graded modules,
{\it Colloq. Math.} \textbf{100}(4), 195--211 (2004).

\bibitem{lundstrom2005}
P. Lundstr\"{o}m,
Crossed product algebras defined by separable extensions,
{\it J. Algebra} {\bf 283} (2005), 723--737.

\bibitem{lundstrom2006}
P. Lundstr\"{o}m,
Strongly groupoid graded rings and cohomology,
{\it Colloq. Math.} \textbf{106}(1), 1--13 (2006).

\bibitem{oinert2010}
P. Lundstr\"{o}m  and  J. \"{O}inert,
Commutativity and Ideals in Category Crossed Products,
{\it Proc. of the Estonian Academy of Sciences}, 
Vol. 59, No. 4, 338--346 (2010).

\bibitem{oinert2012}
 P. Lundstr\"{o}m  and  J. \"{O}inert,
Miyashita action in strongly groupoid graded rings,
{\it Internat. Electronic J. Algebra} {\bf 11}, 46--63 (2012).

\bibitem{lundstrom2012}
P. Lundstr\"{o}m and J. \"{O}inert,
Skew category algebras associated with partially defined dynamical systems,
{\it Internat. Journal of Mathematics}, 
Vol. 23, No. 4, pp. 16 (2012).

\bibitem{N82}
C. N{\v a}st{\v a}sescu and F. Van Oystaeyen, 
{\it Graded Ring Theory}, North-Holland Publishing Co., Amsterdam-New York (1982).

\bibitem{nastasescu1989} C. N{\v a}st{\v a}sescu, M. Van den Bergh and F. Van Oystaeyen, {\it
Separable Functors Applied to Graded Rings}, J. Algebra {\bf
123}, 397-413 (1989).

\bibitem{nastasescu2004}
C. Nastasescu and F. van Oystaeyen,
{\it Methods of graded rings},
Springer Lecture Notes (2004).

\bibitem{Ny2019} P. Nystedt, \emph{A survey of $s$-unital and locally unital rings.}
{\it Revista Integraci\'on.} {\bf 37} (2), 251--260 (2019).  

\bibitem{nystedt2013}
P. Nystedt and J. \"{O}inert,
Simple skew category algebras associated with minimal partially defined dynamical systems,
{\it Discrete and Continuous Dynamical Systems - Series A (DCDS-A)}, 
Vol. 33, No. 9, 4157--4171 (2013).

\bibitem{NyOP22018}  
P. Nystedt, J. \"{O}inert, H. Pinedo, 
Epsilon-strongly groupoid graded rings, the Picard inverse category and cohomology,
{\it Glasgow Math. Journal}, Vol. 62, 233--259 (2020).



\bibitem{reiner1975}
I. Reiner, {\it Maximal Orders}, Academic Press, London, 1975.

\bibitem{ren} J. Renault,
              {\it A Groupoid Approach to $C^*$-algebras},
              Lecture Notes in Mathematics {\bf 793} (1980).

\bibitem{schafer1966}
R. D. Schafer,
{\it An introduction to nonassociative algebras},
Academic Press. New York (1966).

\bibitem{taylor1982}
J. L. Taylor, 
A bigger Brauer group, 
{\it Pacific J. Math.}, Vol. 103, 163--203 (1982).

\bibitem{wisbauer2016}
R. Wisbauer,
Separability in Algebra and Category Theory, 
in {\it Proc. Intern. Conf. Aligarh Muslim University} (2016).

\end{thebibliography}
\end{document}